\documentclass{amsart}     
\usepackage[all, arc]{xy}
\title{What are $E_{\infty}$ ring spaces good for?}

\author{J\,P May}
\address{Department of Mathematics\\
The University of Chicago\\
Chicago, Illinois 60637}
\email{may@math.uchicago.edu}
\urladdr{http://www.math.uchicago.edu/~may}

\usepackage{enumerate}
\usepackage{mathrsfs}

\newtheorem{thm}{Theorem}[section]
\newtheorem{cor}[thm]{Corollary}
\newtheorem{prop}[thm]{Proposition}
\newtheorem{lem}[thm]{Lemma}

\newtheorem{quest}[thm]{Question}

\theoremstyle{definition}
\newtheorem{defn}[thm]{Definition}

\newtheorem{exmp}[thm]{Example}

\theoremstyle{remark}
\newtheorem{rem}[thm]{Remark}

\makeatletter
\let\c@equation\c@thm
\makeatother
\numberwithin{equation}{section}

\DeclareFontFamily{OMS}{rsfs}{\skewchar\font'60}
\DeclareFontShape{OMS}{rsfs}{m}{n}{<-5>rsfs5 <5-7>rsfs7 <7->rsfs10 }{}
\DeclareSymbolFont{rsfs}{OMS}{rsfs}{m}{n}
\DeclareSymbolFontAlphabet{\scr}{rsfs}

\let\overto\xrightarrow

\newcommand{\sA}{\scr{A}}

\newcommand{\sC}{\scr{C}}
\newcommand{\sD}{\scr{D}}
\newcommand{\sE}{\scr{E}}
\newcommand{\sF}{\scr{F}}
\newcommand{\sG}{\scr{G}}

\newcommand{\sI}{\scr{I}}

\newcommand{\sL}{\scr{L}}

\newcommand{\sN}{\scr{N}}
\newcommand{\sO}{\scr{O}}

\newcommand{\sT}{\scr{T}}
\newcommand{\sU}{\scr{U}}

\newcommand{\bC}{\mathbb{C}}

\newcommand{\bE}{\mathbb{E}}
\newcommand{\bF}{\mathbb{F}}

\newcommand{\bM}{\mathbb{M}}
\newcommand{\bN}{\mathbb{N}}

\newcommand{\bR}{\mathbb{R}}

\newcommand{\bZ}{\mathbb{Z}}

\newcommand{\al}{\alpha}
\newcommand{\be}{\beta}

\newcommand{\de}{\delta}
\newcommand{\epz}{\varepsilon}
\newcommand{\ph}{\phi}

\newcommand{\et}{\eta}
\newcommand{\io}{\iota}
\newcommand{\ka}{\kappa}
\newcommand{\la}{\lambda}
\newcommand{\tha}{\theta}

\newcommand{\si}{\sigma}

\newcommand{\ps}{\psi}
\newcommand{\ze}{\zeta}
\newcommand{\om}{\omega}
\newcommand{\GA}{\Gamma}

\newcommand{\DE}{\Delta}
\newcommand{\SI}{\Sigma}

\newcommand{\OM}{\Omega}

\newcommand{\com}{\circ}     
\newcommand{\iso}{\cong}     
\newcommand{\htp}{\simeq}    
\newcommand{\ten}{\otimes}   
\newcommand{\sma}{\wedge}    

\newcommand{\rtarr}{\longrightarrow}

\def\quickop#1{\expandafter\newcommand\csname #1\endcsname{\operatorname{#1}}}
\quickop{Hom} \quickop{End} \quickop{Aut} \quickop{Tel} \quickop{Mic}
\quickop{Ext} \quickop{Tor} \quickop{Id} \quickop{Coker} \quickop{Ker}
\quickop{Lim} \quickop{Colim} \quickop{Holim} \quickop{Hocolim}
\quickop{id} \quickop{tel} \quickop{mic} \quickop{coker}
\quickop{colim} \quickop{holim} \quickop{hocolim} \quickop{im}

\begin{document}

\begin{abstract} Infinite loop space theory, both additive and multiplicative,
arose largely from two basic motivations.  One was to solve calculational questions 
in geometric topology.  The other was to better understand algebraic $K$-theory.
The Adams conjecture is intrinsic to the first motivation, and Quillen's
proof of that led directly to his original, calculationally accessible,
definition of algebraic $K$-theory.   In turn, the infinite loop understanding
of algebraic $K$-theory feeds back into the calculational questions in
geometric topology.  For example, use of infinite loop space theory leads to a 
method for determining the characteristic classes for topological bundles (at odd
primes) in terms of the cohomology of finite groups.   We explain just a little about 
how all that works, focusing on the central role played by $E_{\infty}$ ring spaces.
\end{abstract}

\maketitle

\tableofcontents

\section*{Introduction}

We review and modernize a few of the 1970's applications of $E_{\infty}$ ring
spaces.  We focus on results that involve orientation theory on the infinite
loop space level and on results that involve applications of the $E_{\infty}$ 
ring spaces of algebraic $K$-theory to the analysis of spaces that appear in 
geometric topology.  These $E_{\infty}$ ring spaces arise from bipermutative 
categories. 

Before turning to our main theme, we recall some results of \cite{Class} 
and \cite{MQR} on the classification of bundles and fibrations with additional 
global structure in \S\ref{class}.  We are especially interested in the
classification of bundles and fibrations with a an $R$-orientation for 
some ring spectrum $R$, and use of the (LMS) spectra of \cite{LMS} is the
key to the construction of such classifying spaces.  

We explain how the unit 
$E_{\infty}$ spaces $GL_1R$ and $SL_1R$ of an $E_{\infty}$ ring spectrum $R$ 
relate to the theory of $R$-orientations of bundles and fibrations in \S2.  This use 
of the infinite loop spaces $GL_1R$ was a central theme in the applications of \cite{MQR}, 
where it was crucial to the study of the structure of many spaces of geometric 
interest and to the calculation of their homology and cohomology \cite{CLM}. 
It also provides the foundational starting point for much recent work. 

From the current perspective, \cite{MQR} focused on chromatic level one phenomena and 
their relationship to space level structure, in particular topological bundle theory, 
while recent work focuses on chromatic level two phenomena in stable homotopy theory. 
We illustrate these contrasting points of view in \S\ref{units2} and \S\ref{ThomThom}.  
There are both space level and spectrum level notions of an (infinite loop) $R$-orientation
of a bundle theory, as opposed to an orientation of an individual bundle.
In \S\ref{units2}, we describe universal orientations in terms of the classifying $E_{\infty}$ 
space for $R$-oriented stable bundles and its relationship to other relevant $E_{\infty}$ 
spaces.  In \S\ref{ThomThom}, we reinterpret the theory of universal orientations in terms of
certain $E_{\infty}$ ring Thom spectra $M(G;Y)$ that we construct in \S\ref{Thom}. Geometric 
applications focus on the space 
level theory.  Applications in stable homotopy theory focus on the spectrum level theory.

We illustrate another pair of contrasting points of view implicit in \S1 and \S2 by 
considering the notational tautologies
\[  F = GL_1 S \ \ \ \text{and} \ \ \ SF = SL_1S. \]
Here $F$ is the topological monoid of stable self-homotopy equivalences of spheres and
$SF$ is its submonoid of degree $1$ self-equivalences, while $GL_1S$ and $SL_1S$ are the
unit subspace $Q_1S^0\cup Q_{-1}S^0$ and degree $1$ unit subspace $Q_1S^0$ of the zero$^{th}$ 
space $QS^0 = \colim \OM^nS^n$ of the sphere
spectrum $S$.   The displayed equalities really are tautological, even as $\sL$-spaces and 
thus as infinite loop spaces, where $\sL$ is the linear isometries operad.  Nevertheless, 
we think of the two sides of this tautology very differently.  We claim that $F$ should be
thought of as ``additive'' while $GL_1S$ should be thought of as ``multiplicative''.

The infinite orthogonal group $O$ is a sub-monoid of $F$; we denote the inclusion by $j\colon O\rtarr F$.
On passage to classifying 
spaces we obtain a map of $\sL$-spaces $Bj\colon BO\rtarr BF$. The underlying $H$-space structures on $BO$ 
and $BF$ represent the Whitney sum of vector bundles and the fiberwise smash product of spherical 
fibrations; fiberwise one-point compactification of bundles sends the first to the second.  
The map $Bj$ represents the $J$-homomorphism, and it should be thought of 
as an infinite loop map $BO_{\oplus}\rtarr BF$ since it is the Whitney sum of bundles that gives
rise to the relevant $H$-space structure on $BO$.  Therefore $F$ and $BF$ should be thought of as ``additive''.  

On the other hand, the unit $e\colon S\rtarr R$ of an $E_{\infty}$ ring spectrum $R$ gives a map
of $\sL$-spaces and thus an infinite loop map $GL_1S \rtarr GL_1R$. Here we are thinking
of units of rings under multiplication, and $GL_1 S$ should be thought of as ``multiplicative''.
For example, if we take $R=KO$, then $SL_1R$ is $BO_{\otimes}$; the relevant $H$-space structure
on $BO$ represents the tensor product of vector bundles.  The additive and multiplicative  
$\sL$-space structures on $BO_{\oplus}$ and $BO_{\otimes}$ are quite different and definitely inequivalent;
$BO_{\otimes}$ splits as $BO(1)\times BSO_{\otimes}$, but $BO_{\oplus}$ does not split.  It 
is a deep theorem of Adams and Priddy \cite{AP} that $BSO_{\oplus}$ and $BSO_{\ten}$ are 
actually equivalent as infinite loop spaces, but not by any obvious map and not for any obvious reason.
The analogous statements hold with $O$ and $SO$ replaced by $U$ and $SU$.   

Note that we now have infinite loop maps $SO\rtarr SF=SL_1S\rtarr BO_{\otimes}$.  It turns
out that, after localizing at an odd prime $p$, there are infinite loop spaces $J_{\oplus}$ and 
$J_{\otimes}$ whose homotopy groups are the image of $J$ and there is a diagram of infinite loop maps 
\[ \xymatrix{  
SO \ar[dd] \ar[rd] & & J_{\otimes} \ar[dd]\\
& SF \ar[ur] \ar[dr] & \\
J_{\oplus} \ar[ur] & & BO_{\otimes} \\}\]
such that the composite $J_{\oplus}\rtarr J_{\otimes}$ is an ``exponential'' equivalence of 
infinite loop spaces. This implies that $SF$ splits as an infinite loop space as the product
$J_{\otimes}\times Coker\, J$, where $Coker\, J$ is the fiber of the map $SF\rtarr J_{\otimes}$.
This and related splittings play a fundamental role in calculations in geometric topology,
for example in determining the characteristic classes for stable topological bundles.  

We shall give an outline sketch of how this goes, but without saying anything
about the actual calculations.  Those center around the additive and multiplicative
Dyer-Lashof operations in mod $p$ homology that are induced from the additive and multiplicative $E_{\infty}$ 
structures of $E_{\infty}$ ring spaces.  The distributivity law relating these $E_{\infty}$ 
structures leads to mixed Cartan formulas and mixed Adem relations 
relating these two kinds of operations, and there are Nishida relations relating Steenrod operations
and Dyer-Lashof operations.  Use of such algebraic structure is the only known route for 
understanding the characteristic classes of spherical fibrations and, at odd primes, topological bundles.
It is worth remarking that the analogous structures on generalized homology theories have hardly been studied.

The previous paragraphs concern problems arising from geometric topology.  To explain
the exponential splitting and other key facets of the analysis, we must switch gears and
consider the $E_{\infty}$ ring spaces of algebraic $K$-theory that arise from bipermutative 
categories.  Reversing Quillen's original direction of application, we will thus be considering some 
applications of algebraic $K$-theory to geometric topology. We describe the relevant examples of bipermutative categories and maps between them in \S\ref{exmps}.  The fundamental tool used by
Quillen to relate topological $K$-theory to algebraic $K$-theory is Brauer lifting, and we 
explain the analysis of Brauer lifting on the infinite loop space level in \S\ref{Brauer}.  We 
relate the $K$-theory of finite fields to orientation theory and infinite loop
splittings of geometrically important spaces in \S\ref{last}. 

We hope that this review of just a bit of how $E_{\infty}$ ring theory plays out 
at chromatic level one might help people work out analogous and deeper results 
at higher chromatic levels.  We raise a concrete question related to this. There 
is a mysterious ``rogue object'' (Adams' term \cite[p. 193]{Adams}) that pervades 
the chromatic level one work, namely the infinite loop space $Coker\, J$ mentioned
above. Its first delooping $BCoker\, J$ has a 
natural bundle theoretic interpretation as the classifying space for $j$-oriented 
spherical fibrations, as we shall see in \S\ref{last}, and it is the fundamental 
error term that encodes all chromatic levels greater than one in one indigestible lump.  

As Adams wrote ``to this space or spectrum we consign all of the unsolved problems
of homotopy theory''.  This object seems to be of fundamental interest, but 
it seems to have been largely forgotten.  I'll take the opportunity to explain 
what it is and how it fits into the picture as we understood it in the 1970's. 
As far as I know, we know little more about it now than we did then.  It is natural
to ask the following question.  

\begin{quest}  How precisely does $BCoker\,J$ relate to the chromatic filtration of stable 
homotopy theory?\footnote{Actually, while almost nothing was known about this question when I asked it in the first 
draft of this paper, Nick Kuhn and Justin Noel have obtained a very interesting answer in just the 
last few weeks, that is, in February, 2009.}
\end{quest}

\vspace{2mm}

It is a pleasure to thank Andrew Blumberg for catching errors, suggesting improvements,
and implicitly suggesting that this paper be split off from its prequels \cite{Prequel1, Prequel2}
in this volume.  Andy Baker,
Birgit Richter, and John Lind also caught obscurities and suggested improvements.
Andy and Birgit deserve thanks or complaints for having allowed me to go on at such length.

\section{The classification of oriented bundles and fibrations}\label{class}

For a topological 
monoid $G$, a right $G$-space $Y$, and a left $G$-space $X$, we have the 
two-sided bar construction $B(Y,G,X)$.  It is the geometric realization of
the evident simplicial space with $q$-simplices $Y\times G^q\times X$. We 
fix notations for some maps between bar constructions that we will use 
consistently.
The product on $G$ and its actions on $Y$ and $X$ induce a natural map
\begin{equation}\label{epz} \epz\colon B(Y,G,X)\rtarr Y\times_G X. \end{equation}
The maps $X\rtarr *$ 
and $Y\rtarr *$ induce natural maps
\begin{equation}\label{pq}  p\colon B(Y,G,X) \rtarr B(Y,G,*) \ \ \ 
\text{and}\ \ \ q\colon B(Y,G,X)\rtarr B(*,G,X). \end{equation}
The identifications $Y=Y\times\{\ast\}$ and $X = \{\ast\}\times X$ induce natural maps 
\begin{equation}\label{tq} t\colon Y\rtarr B(Y,G,*) \ \ \text{and} \ \ u\colon X\rtarr B(*,G,X).  
\end{equation}

We let $EG=B(*,G,G)$, which is a free right $G$-space, and $BG = B(*,G,*)$.  We assume 
that the identity element $e\in G$ is a nondegenerate basepoint. We can always arrange 
this by growing a whisker from $e$, but this will only give a monoid 
even when $G$ is a group.  We also assume that $G$ is grouplike, meaning 
that $\pi_0(G)$ is a group under the induced product.  When $G$ is a group,
it is convenient to assume further that $G$ acts effectively on $X$.  Recall
that for any such $X$ the associated principal bundle functor and the functor that sends a
principal $G$-bundle $P$ to $P\times_G X$ give a natural bijection between
the set of equivalence classes of principal $G$-bundles and the set of
equivalence classes of $G$-bundles with fiber $X$.

We recall from \cite{Class} how the bar construction is used to classify bundles and fibrations. 
To begin with, the following diagram is a pullback even when $G$ is just a monoid.
\[ \xymatrix{
 B(Y,G,X) \ar[r]^-{q} \ar[d]_{p} & B(*,G,X) \ar[d]^{p} \\
B(Y,G,*) \ar[r]_-{q} & BG \\} \]

When $G$ is a topological group, $p\colon EG\rtarr BG$ is a 
(numerable) universal principal $G$-bundle.  In fact, $EG$ is also a topological group, with
$G$ as a closed subgroup, and $BG$ is the homogeneous space $EG/G$ of right cosets.   The map $p\colon B(*,G,X)\rtarr BG$ is the associated
universal $G$-bundle with fiber $X$.  The map $p\colon B(Y,G,X)\rtarr B(Y,G,*)$ is a $G$-bundle
with fiber $X$, and it is classified by $q\colon B(Y,G,*)\rtarr BG$.  If $G$ acts principally on $Y$ and effectively on $X$, then the following diagram is a pullback in which the maps $\epz$ are (weak) equivalences.
\[ \xymatrix{
 B(Y,G,X) \ar[r]^-{\epz} \ar[d]_{p} & Y\times_GX \ar[d]^{p} \\
B(Y,G,*) \ar[r]_-{\epz} & Y/G \\} \]
The classification theorem for bundles states that for any $X$ the set $[A,BG]$ of (unbased) homotopy 
classes of maps $A\rtarr BG$ is naturally isomorphic to the set of equivalence classes of $G$-bundles
with fiber $X$ over $A$ when $A$ has the homotopy type of a CW complex.  Pullback of the 
universal bundle gives the map in one direction.  In the other direction, for a principal $G$-bundle 
$Y\rtarr A$, the two pullback squares above combine to give the classifying map 
\[ \xymatrix@1{ A\iso Y/G \ar[r]^-{\epz^{-1}} &  B(Y,G,*) \ar[r]^-{q} & BG,\\}  \]
where $\epz^{-1}$ is any chosen (right) homotopy inverse to $\epz$.  See \cite[\S\S8,9]{Class} for details and proofs. However, we point out for later reference one fact that drops out of the proof. Consider the diagram
\begin{equation}\label{later}
\xymatrix@1{
BG & B(EG,G,*) \ar[l]_-{\epz} \ar[r]^-{q} & BG. \\}
\end{equation}
For any chosen homotopy inverse $\epz^{-1}$, $q\com \epz^{-1}$ is homotopic to the identity.

When $G$ is only a monoid, one has to develop a theory of principal and associated fibrations.
Also, the maps in our first pullback diagram are then only quasifibrations, and we have to replace
quasifibrations by fibrations since pullbacks of quasifibrations need not be quasifibrations.
Once these details are taken care of, the classification of fibrations works in the same way
as the classification of bundles. Taking $G$ to be the monoid $F(X) = hAut(X)$ of based self 
homotopy equivalences of a based CW complex $X$, $BF(X)$ classifies well-sectioned (the section
is a fiberwise cofibration) fibrations with fiber $X$.  Letting $SF(X)$ be the submonoid of 
self-maps homotopic to the identity and
defining orientations appropriately, $BSF(X)$ classifies oriented well-sectioned fibrations 
with fiber $X$.  See \cite[\S9]{Class} and \cite[\S1]{Fib} for details and proofs.  

We are interested in the role of $Y$ in the constructions above.  We have already exploited the 
variable $Y$ in our sketch proof of the classification theorem, 
but it has other uses that are of more direct interest to us here.  A general theory of 
$Y$-structures on bundles and fibrations is 
given in \cite[\S10]{Class}.  For
simple examples, consider a map $f\colon H\rtarr G$ of topological monoids.
In the generality of monoids, it is sensible, although non-standard, to define 
\begin{equation}\label{ersatz} 
G/H = B(G,H,*) \ \ \ \text{and} \ \ \ H\backslash G = B(*,H,G);
\end{equation}
these are consistent up to homotopy with the usual notions when $H$ and $G$ are groups. 
With our assumption that $H$ and $G$ are grouplike, both of these are equivalent to the fiber 
of $Bf\colon BH\rtarr BG$.  As explained in \cite[10.3, 10.4]{Class}, with the notations of
(\ref{ersatz}), the theory of 
$Y$-structures specializes to show that $G/H$ classifies $H$-fibrations 
with a trivialization as a $G$-fibration and $B(H\backslash G,G,*)$ classifies 
$G$-fibrations with a ``reduction of the structural monoid'' from $G$ to $H$.

However, our main interest is in the specialization of the theory of $Y$-structures
to the classification of oriented fibrations and bundles that is explained in \cite[Ch. III]{MQR},
where more details may be found. 

Recall the language of functors with cartesian product (FCP's) from \cite[\S\S2, 12]{Prequel1}.
As there, we understand FCP's and FSP's to be commutative in this paper.  We are concerned 
specifically with the monoid-valued $\sI$-FCP's and $\sI_c$-FCP's of \cite[\S2]{Prequel1}.  Of course, group-valued $\sI$-FCP's and $\sI_c$-FCP's are defined similarly.
Remember that the categories of monoid or group-valued $\sI$-FCP's and $\sI_c$-FCP's 
are equivalent. 

For finite dimensional inner product spaces $V$, let $F(V) = F(S^V)$ and 
$SF(V) = SF(S^V)$.  For countably infinite dimensional $U$, $F(U)$ 
and $SF(U)$ are defined by passage to colimits over the inclusions $V\subset W$ of 
finite dimensional subspaces of $U$.  These are $\sI_c$-FCP's via smash products of 
maps of spheres, and they are monoid-valued under composition. To avoid ambiguity, 
it would be sensible to write 
$F$ and $SF$ only for these monoid-valued FCP's 
and to write $GL_1(S)$ and $SL_1(S)$ for their values on $\bR^{\infty}$, but we shall allow the 
alternative notations $F = F(\bR^{\infty})$ and $SF=SF(\bR^{\infty})$, as in the introduction.  
{\em We agree to use the notations 
$F$ and $SF$ when we are 
thinking about the roles of these spaces in space level applications and the notations
$GL_1(S)$ and $SL_1(S)$ when we are thinking about their role in stable homotopy theory.}
One point is that it is quite irrelevent to the space level applications that the spaces
$F$ and $SF$ happen to be components of the zero$^{th}$ space of the sphere spectrum.

Throughout the rest of this section and the following three sections, we let $G$ be a monoid-valued 
$\sI_c$-FCP together with a map $j\colon G\rtarr F$ of monoid-valued $\sI_c$-FCP's.  We  
assume that $G$ is grouplike, meaning that each $\pi_0(G(V))$ is a group.  The letter $j$
is a reminder of the $J$-homomorphism, which it induces when $G=O$. With a little interpretation 
(such as using complex rather than real inner product spaces) examples
include $O$, $SO$, $Spin$, $String$, $U$, $SU$, $Sp$, $Top$, and $STop$.  We also 
write $G$ for $G(\bR^{\infty})$, despite the ambiguity, and we agree to write $SG$ for the component 
of the identity of $G$ even when $G$ is connected and thus $SG=G$. The classifying spaces
$BG(V)$ give an $\sI_c$-FCP, but of course it is not monoid-valued. More generally, there 
is an evident notion of a (left or right) action of a monoid-valued functor $\sI_c$-FCP
on an $\sI_c$-FCP. The following observation \cite[II.2.2]{MQR} holds by the product-preserving 
nature of the two-sided bar construction.

\begin{prop}\label{twoside1} If $G$ is a monoid-valued $\sI_c$-FCP that acts from the right 
and left on $\sI_c$-FCP's $Y$ and $X$, then the functor $B(Y,G,X)$ 
specified by
\[  B(Y,G,X)(V) = B(Y(V),G(V),X(V)) \]
inherits a structure of an $\sI_c$-FCP from $G$, $Y$, and $X$. 
\end{prop}

We can think of $G$-bundles or $G$-fibrations, which by abuse we call $G$-bundles
in what follows, as $F$-fibrations with a reduction of their 
structural monoids to $G$.  Here we are thinking of finite dimensional inner product spaces,
and we understand the fibers of these bundles to be spheres $S^V$.  The maps on classifying spaces induced 
by the product maps $G(V)\times G(W)\rtarr G(V \oplus W)$ of the FCP are covered by maps
$Sph(V)\sma Sph(W)\rtarr Sph(V\oplus W)$ of universal spherical bundles.  The
whole structure in sight forms a PFSP (parametrized functor with smash product), as
specified in \cite[Ch. 23]{MS}.  That point of view best captures the relationships
among FCP's, FSP's, and Thom spectra, but we shall not go into that here.

Now recall that $GL_1R$ is the space of unit components of the $0^{th}$ space $R_0$ of a 
commutative ring spectrum $R$ and that $SL_1R$ is the component of the identity. The
space $GL_1R$ has a right action by the monoid $F$.\footnote{The notation $FR$ for
$GL_1R$ originally used in \cite{MQR} emphasizes this relationship to $F$.} This is a 
trivial observation, but a very convenient one that is not available with other definitions 
of spectra.  Indeed, $R_0$ is homeomorphic to $\OM^V R(V)$ and, since $F(V) = F(S^V)$, composition 
of maps gives a right action of $F(V)$ on $\OM^V R(V)$.  When $R$ is an up-to-homotopy commutative 
ring spectrum, this action restricts to an action of $F(V)$ on $GL_1R$ and of $SF(V)$ on 
$SL_1R$.  These actions are compatible with colimits and therefore induce a right action
of the monoid $F$ on the space $GL_1R$ and of $SF$ on $SL_1R$. These actions pull back 
to actions by the monoids $G$ and $SG$.  

An $R$-orientation of a well-sectioned bundle $E\rtarr B$ with fiber $S^V$ is a cohomology class 
of its Thom space $E/B$ that restricts to a unit on fibers.  Such a class is represented by a 
map $E/B\rtarr R(V)$.  Taking $B$ to be connected, a single fiber will do, and then the 
restriction is a based map $S^V\rtarr R(V)$ and thus a based map $S^0\rtarr \OM^VR(V)\iso R_0$.  
The image of $1$ must be a point of $GL_1(R)$.  This should give a hint as to why the following 
result from \cite[I\S2]{MQR} is plausible.

\begin{thm} The space $B(GL_1R, G(V), *)$ classifies equivalence classes of $R$-oriented
$G(V)$-bundles with fiber $S^V$.
\end{thm}

\begin{cor} The space $B(SL_1R, SG(V), *)$ classifies $R$-oriented $SG(V)$-bun\-dles with
fiber $S^V$.  
\end{cor}

The interpretation requires a bit of care.  Orientations depend only on the connective cover
of $R$, so we may assume that $R$ is connective. An $R$-oriented bundle inherits a $k$-orientation,
where $k=\pi_0(R)$.  We specify $R$-orientations by requiring them to be consistent with
preassigned $k$-orientations.  Precisely, the $k$-orientation prescribes a Thom class in 
$H^n(T\xi;k)\iso R^n(T\xi)$ for an $n$-dimensional $G(V)$-bundle $\xi$, and we 
require an $R$-orientation to restrict on fibers to the resulting fundamental classes.
An $SG(V)$-bundle is an integrally oriented $G(V)$-bundle, and we define an
$R$-oriented $SG(V)$-bundle to be an $R$-oriented $G(V)$-bundle and an $SG(V)$-bundle
whose prescribed $k$-oriention is that induced from its integral orientation.

Along with these classifying spaces, we have Thom spectra associated to bundles 
and fibrations with $Y$-structures, such as orientations \cite[IV.2.5]{MQR}. 
We discuss $E_{\infty}$-structures on classifying spaces and on Thom spectra in
the following two sections.  

\section{$E_{\infty}$ structures on classifying spaces and orientation theory}\label{units}

Still considering a monoid-valued grouplike $\sI_c$-FCP $G$ over $F$, 
we now assume further that $R$ is a (connective) $E_{\infty}$ ring spectrum
and focus on the stable case, writing $G$ and $SG$ for $G(\bR^{\infty})$ and 
$SG(\bR^{\infty})$. The analogues for stable bundles of the 
classification results above remain valid, but we now concentrate on
$E_{\infty}$ structures on the stable classifying spaces. 

Recall from \cite[\S2]{Prequel1} that we have a functor from $\sI$-FCP's, or equivalently 
$\sI_c$-FCP's, to $\sL$-spaces.  
For any operad $\sO$, such as $\sL$, the category $\sO[\sT]$ of $\sO$-spaces has 
finite products, so it also makes sense to define monoids and groups in the category 
$\sO[\sT]$.  For a monoid $G$ in $\sO[\sT]$, the monoid product and the product induced 
by the operad action are homotopic \cite[3.4]{MayPer}.  It also makes sense to define 
left and right actions of $G$ on $\sO$-spaces.  The functors from $\sI$-FCP's to $\sI_c$-FCP's 
to $\sL$-spaces are product preserving and so preserve monoids, groups, and their actions.
Moreover, we have the following analogue of Proposition \ref{twoside1}.

\begin{prop}\label{BYGXinf} If $G$ is a monoid in $\sO[\sT]$ that acts from the right 
and left on $\sO$-spaces $Y$ and $X$, then $B(Y,G,X)$ inherits an $\sO$-space structure
from $G$, $Y$, and $X$.  In particular, $BG$ is an $\sO$-space. Moreover, the natural map 
$\ze\colon G\rtarr \OM BG$ 
is a map of $\sO$-spaces and a group completion.
\end{prop}
\begin{proof} $B(Y,G,X)$ is the geometric realization of a simplicial $\sO$-space and
is therefore an $\sO$-space.  The statements about $\ze$ hold by \cite[3.4]{MayPer} and 
\cite[15.1]{Class}.
\end{proof}

As we reproved in \cite[Corollary B.4]{Prequel1}, when $\sO$ is an $E_{\infty}$ operad 
this implies that the first delooping $\bE_1 G$ is equivalent to $BG$ as an $\sO$-space.
Said another way, the spectra obtained by applying the additive infinite loop space machine 
$\mathbf{E}$ to $G\htp \OM BG$ are equivalent to those obtained by applying $\OM \mathbf{E}$ 
to $BG\htp \mathbf{E}_0 BG$. 

When $Y$ and $X$ are $\sI_c$-FCP's with right and left actions by $G$, the $\sL$-space
structure of Proposition \ref{BYGXinf} is the same as the $\sL$-space structure obtained 
by passage to colimits from the $\sI_c$-FCP structure on $B(Y,G,X)$ of Proposition \ref{twoside1}.
This does not apply to the right $F$-space $Y=GL_1R$ for an $\sL$-spectrum $R$, but in that case 
we can check from the definition of an $\sL$-prespectrum \cite[\S5]{Prequel1} that the action map
$GL_1R\times F\rtarr GL_1R$ is a map of $\sL$-spaces (see \cite[p.\,80]{MQR}). 

We conclude that, in the stable case, the spaces $B(Y,G,X)$ that we focused on in
the previous section are grouplike $\sL$-spaces and therefore, by the additive infinite
loop space machine, are naturally equivalent to the $0^{th}$ spaces of associated spectra.  
Thus we may think of them as infinite loop spaces.  This result and its implications were 
the main focus of \cite{MQR} and much of \cite{CLM}, where the homologies of many of these 
infinite loop spaces are calculated in detail by use of the implied Dyer-Lashof homology 
operations.

Taking $X=*$ and thus focusing on classifying spaces, it is convenient to abbreviate notation 
by writing
\[ B(Y,G,*) = B(G;Y) \] 
for the classifying space of stable $G$-bundles equipped with $Y$-structures.
It comes with natural maps 
\[ t\colon Y\rtarr B(G;Y)\ \ \ \ \text{and} \ \ \ \ q\colon B(G;Y)\rtarr BG. \]  
When we specialize to $Y = SL_1R$ or $Y=GL_1R$, we abbreviate further by writing
\[  B(SG;SL_1R) = B(SG;R) \ \ \ \ \text{and} \ \ \ \  B(G;GL_1R) = B(G;R).\] 
These are the classifying spaces for stable $R$-oriented $SG$-bundles and stable
$R$-oriented $G$-bundles.  It is important to remember that these spaces depend only on 
$GL_1R$, regarded as an $F$-space and an $\sL$-space, and not on the spectrum $R$; that is, 
they are space level constructions.  With these notations, the discussion above leads to the following result, which is \cite[IV.3.1]{MQR}.

\begin{thm}\label{ordiag} let $R$ be an $\sL$-spectrum and let $\pi_0(R) = k$. Then all
spaces are grouplike $\sL$-spaces and all maps are $\sL$-maps in the following ``stable
orientation diagram''.  It displays two maps of fibration sequences.
\[ \xymatrix{
SG \ar[r]^-{e} \ar[d] & SL_1R \ar[r]^-{t} \ar[d] &  B(SG;R) \ar[r]^-{q} \ar[d]& BSG \ar[d]\\
G \ar[r]^-{e} \ar@{=} [d] & GL_1R \ar[r]^-{t} \ar[d]^{d} &  B(G;R) \ar[r]^-{q} \ar[d]^{Bd} 
& BG \ar@{=}[d]\\
G \ar[r]_-{de}  & GL_1(k) \ar[r]^-{t}  &  B(SG;k) \ar[r]^-{q} & BG \\} \]
The diagram is functorial in $R$; that is, a map $R\rtarr Q$ of $\sL$-spectra induces
a map from the diagram of $\sL$-spaces for $R$ to the diagram of $\sL$-spaces 
for $Q$.
\end{thm}

The unstable precursor (for finite dimensional $V$) and its bundle theoretic interpretation
are discussed in \cite[pp. 55-59]{MQR}. The top vertical arrows are inclusions and the map 
$d$ is just discretization.  The maps $e$ are induced by the unit $S\rtarr R$.
On passage to $0^{th}$-spaces, the unit gives a map $F=GL_1S\rtarr GL_1 R$, and we are 
assuming that we have a map $j\colon G\rtarr F$. We continue to write $e$ for the composite
$e\com j$.  Writing $BGL_1R$ for the delooping of $GL_1R$ given by the additive infinite loop
space machine, define a generalized first Stiefel-Whitney class by
\[ w_1(R) = Be\colon BG\rtarr BGL_1R.\]   
Then $w_1(R)$ is the universal obstruction to giving a stable $G$-bundle an $R$-orientation;
see \cite[pp. 81--83]{MQR} for discussion. The map $t$ 
represents the functor that sends a unit of $R^0(X)$ to the trivial $G$-bundle over $X$ 
oriented by that unit. The map $q$ represents the functor that sends an $R$-oriented
stable $G$-bundle over $X$ to its underlying $G$-bundle, forgetting the orientation.

There is a close relationship between orientations and trivializations that plays a major
role in the applications of \cite{CLM, MQR}.  We recall some of it here, athough it is 
tangential to our main theme.   The following result is the starting point.
Its unstable precursor and bundle theoretic interpretation are discussed in 
\cite[pp. 59-60]{MQR}.   It and other results to follow have analogues in the 
oriented case, with $G$ and $F$ replaced by $SG$ and $SF$.

\begin{thm}\label{compdiag1}  Let $R$ be an $\sL$-spectrum. 
Then all spaces in the left three squares are grouplike $\sL$-spaces and 
all maps are $\sL$-maps in the following diagram.  It displays a map of
fibration sequences, and it is natural in $R$.
\[ \xymatrix{
G \ar[r]^{j} \ar@{=}[d] & F \ar[r]^-{t} \ar[d]^{e}  & F/G \ar[r]^-{q} \ar[d]^{Be} 
& BG \ar@{=}[d] \ar[r]^-{Bj} & BF \ar[d]^{Be}\ar[r] & \cdots\\
G \ar[r]_-{e} & GL_1R \ar[r]_-{t} & B(G;R) \ar[r]_-{q} & BG \ar[r]_-{w_1(R)} 
& BGL_1R \ar[r] & \dots\\} \]
\end{thm}

The left map labeled $Be$ is $B(e,\text{id},\text{id})\colon B(F,G,*) \rtarr B(GL_1R,G,*)$. 
Since the first three squares of the diagram are commutative diagrams of $\sL$-spaces, 
we get the fourth square from the induced fibration of spectra.  It relates the $J$-map $Bj$, 
which is the universal obstruction to 
the $F$-trivialization of $G$-bundles, to the universal obstruction $w_1(R) = Be$ to the 
$R$-orientability of $G$-bundles.  In particular, it gives a structured interpretation of the fact 
that if a $G$-bundle is $F$-trivializable, then it is $R$-orientable for any $R$.   

The previous result works more generally with $F$ replaced by any $G'$ between $G$ 
and $F$, but we focus on $G'= F$ since that is the case of greatest interest.  We state the 
following analogue for Thom spectra in the general case. Its proof falls directly out of the
definitions \cite[IV.2.6]{MQR}. However, for readability, we agree to start with $H\rtarr G$ 
rather than $G\rtarr G'$, in analogy with the standard convention of writing $H$ for a generic 
subgroup of a group $G$.  We are thinking of the case $R=MH$ in Theorem \ref{compdiag1}.  The case 
$G = F$ plays a key role in Ray's study \cite{Ray} of the bordism $J$-homomorphism.

\begin{prop}\label{bored} Let $i\colon H\rtarr G$ be a map of grouplike
monoid-valued $\sI_c$-functors over $F$. Then there is a map
of $\sL$-spaces $j\colon H\backslash G\rtarr GL_1(MH)$ that concides 
with $j\colon G\rtarr F$ when $H=e$ and makes the following diagrams 
of $\sL$-spaces commute. 
\[ \xymatrix{  
G\ar[r]^-{u} \ar[dr]_{e} & H\backslash G \ar[d]^{j} \\
& GL_1(MH)  \\}
\ \ \ \text{and} \ \ \
\xymatrix{
B(H;H\backslash G) \ar[d]_{Bj} \ar[r]^-{q} & BH \\
B(H;MH) \ar[ur]_-{q} \\} \]
\end{prop}

\section{Universally defined orientations of $G$-bundles} \label{units2}

Universally defined canonical orientations of $G$-bundles are of central importance
to both the early work of the 1970's and to current work, and we shall discuss them
in this section and the next. The early geometric examples are the 
Atiyah--Bott--Shapiro $kO$-orientations of $Spin$-bundles and $kU$-orientations 
of $U$-bundles and the Sullivan (odd primary) spherical $kO$-orientations 
of $SPL$-bundles.  We are interested in stable bundles and in the relationship
of their orientations to infinite loop space theory and to stable homotopy theory.
We fix an $E_{\infty}$ ring spectrum $R$.

There are two homotopical ways of defining and thinking about such universally defined 
orientations of $G$-bundles, one on the classifying space level and the 
other on the Thom spectrum level.   The main focus of \cite{MQR} was on 
the classifying space level and calculational applications to 
geometric topology.  The main modern focus is on the Thom spectrum level 
and calculational applications to stable homotopy theory, as in 
\cite{five, AHS, Blum1, BCS}.  We work on the classifying space level here and 
turn to the Thom spectrum level and the comparison of the two in the next section.

\begin{defn}\label{defBorient} 
An $R$-orientation of $G$ is a map of $H$-spaces $g\colon BG\rtarr B(G;R)$
such that $q\com g = id$ in the homotopy category of spaces.  A spherical
$R$-orientation is a map of $H$-spaces $g\colon BG\rtarr B(F;R)$ such 
that the following diagram commutes in the homotopy category.
\[  \xymatrix{
   BG\ar[dr]_{Bj} \ar[r]^-{g} & B(F;R) \ar[d]^q \\
   & BF \\} \]
We call these $E_{\infty}$ $R$-orientations
if $g$ is a map of $\sL$-spaces such that $q\com g = id$ or $q\com g = Bj$ 
in the homotopy category of $\sL$-spaces.  
\end{defn}

There is a minor technical nuisance that perhaps should be pointed out but should 
not be allowed to interrupt the flow. In practice, instead of 
actual maps $g$ of $\sL$-spaces as in the definition, we often encounter 
diagrams of explicit $\sL$-maps of the form
\[ \xymatrix@1{
X & X' \ar[l]_-{\epz} \ar[r]^{\nu} & Y \\} \]
in which, ignoring the $\sL$-structure, $\epz$ is a weak equivalence. The category
of $\sL$-spaces has a model structure with such maps $\epz$ as the weak equivalences,
hence such a diagram gives a well-defined map in the homotopy category of $\sL$-spaces.
We agree to think of such a diagram with $X = BG$ and $Y=B(G;R)$
as an $E_{\infty}$ $R$-orientation.

With the notation of (\ref{ersatz}), $G\backslash G = EG$.  This is a contractible 
space, and $\epz\colon EG\rtarr *$ is 
both a $G$-map and a map of $\sL$-spaces.  The case $H=G$ of Proposition \ref{bored}, 
together with (\ref{later}), gives a structured reformulation 
of the standard observation that $G$-bundles have tautological $MG$-orientations.

\begin{cor}\label{BTaut}  The following diagram of $\sL$-spaces commutes, its map
$\epz$ and top map $q$ are equivalences, and $q\com \epz^{-1} = \text{id}$ in the homotopy 
category of spaces.
\[ \xymatrix{
BG & B(EG,G,*)=B(G;G\backslash G) \ar[l]_-{\epz} \ar[d]^-{Bj} \ar[r]^-{q} & BG \\
& B(GL_1(MG),G,*)=B(G;MG) \ar[ur]_-{q} &  \\} \]
\end{cor}

The following two direct consequences of Definition \ref{defBorient}  are \cite[V.2.1, V.2.3]{MQR}.
Let $\ph\colon G\times G\rtarr G$ be the product and $\chi\colon G\rtarr G$ be the
inverse map of $G$. In the cases we use, both are given by maps 
of $\sL$-spaces.

\begin{prop}\label{split}  If $g\colon BG\rtarr B(G;R)$ is an $R$-orientation, then
the composite 
\[ \xymatrix@1{ 
GL_1R\times BG \ar[r]^-{t\times g} & B(G;R)\times B(G;R)\ar[r]^-{\ph} & B(G;R)\\} \]
is an equivalence of $H$-spaces.  If $g$ is an $E_{\infty}$ 
$R$-orientation, then $\ph\com (t\times g)$ is an equivalence of infinite loop spaces.
\end{prop}  

\begin{thm}\label{compdiag2} An $E_{\infty}$ spherical $R$-orientation $g\colon BG\rtarr B(F;R)$
induces a map $f$ such that the following is a commutative diagram of infinite loop spaces.  
It displays a map of fibration sequences. 
\[  \xymatrix{
F \ar[r]^-{t} \ar[d]_{\chi}  & F/G \ar[r]^-{q} \ar[d]^{f} 
& BG \ar[d]^{g} \ar[r]^-{Bj} & BF \ar@{=}[d] \ar[r]^{Bt} & B(F/G) \ar[d]^{Bf} \ar[r] & \cdots \\
F \ar[r]_-{e} & GL_1R \ar[r]_-{t} & B(F;R) \ar[r]_-{q} & BF \ar[r]_-{Be} 
& BGL_1R \ar[r] & \cdots\\} \]
\end{thm} 

The third square is a factorization of the $J$-homomorphism map $Bj$. It is used in
conjunction with the following observation, which is 
\cite[V.2.2]{MQR}. For a grouplike $H$-space $X$ and $H$-maps $\al,\be\colon X\rtarr X$, 
we define $\al/\be = \ph(\al\times \chi \be)\DE\colon X\rtarr X$; when we think of $X$ as 
an ``additive'' $H$-space, we write this as $\al - \be$.  These are infinite loop maps when $\al$ 
and $\be$ are infinite loop maps.   

\begin{prop}\label{cannibal} An (up to homotopy) map of ring spectra $\ps\colon R\rtarr R$ 
induces a map $c(\psi)$ such that the following diagram is homotopy
commutative.
\[ \xymatrix{
     GL_1R \ar[d]_{\psi/1} \ar[r]^{t} & B(G;R)  \ar[dl]_{c(\ps)} \ar[d]^{(B\ps)/1}\\\
     GL_1R \ar[r]_{t} & B(G;R) \\} \]
If $R$ is an $\sL$-spectrum and $\ps$ is a map of $\sL$-spectra, then the diagram is a
homotopy commutative diagram of maps of infinite loop spaces.
\end{prop}

The intuition is that $c(\ps)$ is given by taking the quotient $(\ps\com \mu)/\mu$ of
an orientation $\mu$ and the twisted orientation $\ps\com \mu$ to obtain a unit of $R$.

In the applications of this result, it is crucial to apply the last sentence to the Adams operations 
$\ps^r\colon kO\rtarr kO$ but, even at this late date, I would not know how to justify that 
without knowing about bipermutative categories, algebraic $K$-theory, and the relationships 
among bipermutative 
categories, $E_{\infty}$ ring spaces, and $E_{\infty}$ ring spectra.  Perhaps the deepest work in 
\cite{MQR}, joint with Tornehave, provides such a justification by using Brauer lifting to 
relate the algebraic $K$-theory of the algebraic closures $\bar{\bF}_q$ of finite fields to 
topological $K$-theory on the multiplicative infinite loop space level.
The proof makes essential use of the results described in \cite[\S10]{Prequel1}
on the localization of unit spectra $sl_1R$. I'll describe how the argument goes in 
\S\ref{Brauer}.

As noted, the main geometric examples are the Atiyah--Bott--Shapiro orientations and the 
Sullivan (odd primary) orientation. For the latter, work of Kirby-Siebenmann \cite{KS}) 
shows that $BSPL$ is equivalent to $BSTop$ away from $2$. This is a major convenience 
since $STop$ fits into our framework of monoid-valued $\sI_c$-FCP's
and $SPL$ does not.  Using deep results of Adams and Priddy \cite{AP} and 
Madsen, Snaith, and Tornehave \cite{MST}, I proved the following result in
\cite[V.7.11, V.7.16]{MQR} by first constructing an infinite loop map $f$ such that the
left square commutes in the diagram of Theorem \ref{compdiag2} and then constructing $g$.

\begin{thm}\label{spherical} Localizing at a prime (odd in the case of $STop$), 
the Atiyah--Bott--Shapiro $kO$-orientation of $Spin$ and $kU$-orientation of 
$U$ and the Sullivan $kO$-orientation of $STop$ are $E_{\infty}$ spherical 
orientations.
\end{thm}

This result, together with Friedlander's proof of the complex Adams conjecture on the infinite
loop space level \cite{Fried}, leads to an analysis on that level of the work of Adams
on the $J$-homomorphism \cite{JX1, JX2, JX3, JX4} and the work of Sullivan on the structure of $BSTop$
(alias $BSPL$) \cite{Sull}.  I'll resist the temptation to give a full summary of 
that work here. The relevant part of \cite{MQR}, its Chapter V, is more readable and less 
dated notationally than most of the rest of that volume.  It chases diagrams built up from those
recorded above, with $G = Spin$ or $G=STop$ and $R=kO$, to show how to split all spaces in 
sight $p$-locally into pieces that are entirely understood in terms of $K$-theory and 
the space $BCoker\, J$, whose homotopy groups are the cokernel of the $J$-homomorphism
\[ (Bj)_* \colon \pi_*(BO) \rtarr \pi_*(BF). \]
At $p>2$, the space $BCoker\, J$ can be defined to be the fiber of the map
\[ c(\ps^r)\colon B(SF;kO) \rtarr SL_1(kO) = BO_{\otimes},\] 
where $r$ is a unit mod $p^2$. At $p=2$, one should take $r=3$ and replace 
$BO_{\otimes}$ by its $2$-connected cover $BSpin_{\otimes}$ in this definition,
and the description of the homotopy groups of $BCoker\,J$ requires a well 
understood small modification.

However, the fact that $BCoker\, J$ is an infinite loop space comes from the work 
using Brauer lifting that I cited above.  In fact, it turns out that, for $p$ odd,
$BCoker\, J$ is equivalent to $B(SF;K({\bF}_{r}))$, where $r$ is a prime power $q^a$ 
that is a unit mod $p^2$ and $\bF_r$ is the field with $r$ elements.  This is an 
infinite loop space because $K({\bF}_{r})$ is an $E_{\infty}$ ring spectrum.
There is an analogue for $p=2$.  

At odd primes, this description of $BCoker\, J$ is consistent with the description of
$Coker\,J$ alluded to in the introduction and leads to the splitting of $BSF$ as 
$BJ\times BCoker\, J$ as an infinite loop space; here $BJ$ is equivalent to the 
infinite loop space $SL_1K({\bF}_{r})$.  The proof again makes essential use of 
the results on spectra of units described in \cite[\S10]{Prequel1}. I'll sketch how 
this argument goes in \S\ref{last}.

The definitive description of the infinite loop structure on $BSTop$ is given in 
\cite{IMon}, where a consistency statement about infinite loop space machines that is not 
implied by May and Thomason \cite{MT} plays a crucial role in putting things together; 
it is described at the end of \cite[\S11]{Prequel2}.  Part of the conclusion is that, at an odd 
prime $p$, the infinite loop space $BSTop$ is equivalent to $B(SF;kO)$ and splits as 
$BO\times BCoker\, J$.  We will say a little bit about this in \S\ref{last}. 

Joachim \cite{Joa} (see also \cite{five}) has recently proved a Thom spectrum level 
result which implies the following result on the classifying space level.  It 
substantially strengthens the Atiyah--Bott--Shapiro part of Theorem \ref{spherical}. 

\begin{thm}\label{Joach} The Atiyah--Bott--Shapiro orientations 
\[  g\colon BSpin \rtarr B(Spin;kO) \ \ \text{and} \ \ g\colon BU\rtarr B(U;kO) \]
are $E_{\infty}$ orientations.
\end{thm}

\section{$E_{\infty}$ ring structures on Thom spectra $M(G;Y)$}\label{Thom}

Let $R$ be an $\sL$-spectrum throughout this section and the next. 
Two different constructions of Thom $\sL$-spectra $M(G;R)$ are given in  
\cite[IV.2.5, IV.3.3]{MQR}, and they are compared in \cite[IV.3.5]{MQR}.
They are specializations of more general constructon of $\sL$-spectra
$M(G;Y)$. We first describe the second construction. We then describe the first 
construction in the modern language of \cite[Ch. 23]{MS} and interpolate
some commentary on the modern perspective on these constructions. 

In the previous section, we focused on spaces $B(G;Y)\equiv B(Y,G,*)$ arising from a 
group\-like monoid-valued $\sI$-FCP $G$ over $F$ and an $\sL$-space $Y$. Here
$G=G(\bR^{\infty})$ is the union over finite dimensional $V\subset \bR^{\infty}$
of the $G(V)$.  We now Thomify from that perspective, using the passage from $\sL$-prespectra
to $\sL$-spectra recalled in \cite[\S5]{Prequel1}.  Recall from \S2 that 
the $S^V$ give the  sphere $\sI$-FSP $S$.  We have Thom spaces
\[ T(G;Y)(V) = B(Y,G(V),S^V)/B(Y,G(V),\infty) \]
where $\infty\in S^V$ is the point at $\infty$.  
Smashing with $S^W$ for $W$ orthogonal to $V$ and moving it inside the bar
construction, as we can do, we see that the identifications $S^V\sma S^W\iso S^{V\oplus W}$
and the inclusions $G(V)\subset G(V\oplus W)$ induce structure maps 
\[ \si\colon T(G;Y)(V)\sma S^W \rtarr T(G;Y)(V\oplus W). \]
For $f\in \sL(j)$, we have maps of prespectra
\[ \xi_j\colon T(G;Y)^{[j]} \rtarr f^*T(G;Y) \]
Explicitly, abbreviating $T(G;Y) = T$, the required maps 
\[ \xi_j\colon T(V_1)\sma\cdots \sma T(V_j)\rtarr T(f(V_1\oplus\cdots\oplus V_j)) \]
are obtained by identifying $T(V_1)\sma\cdots \sma T(V_j)$ with a quotient of
\[ B(Y^j,G(V_1) \times \cdots \times G(V_j), S^{V_1\oplus \cdots \oplus V_j}) \]
and then applying $B(\xi_j(f),G(f)\com \om, S^f)$, where the map $\xi_j(f)\colon Y^j\rtarr Y$
is given by the operad action on $Y$, the map
\[ G(f)\com \om\colon G(V_1) \times \cdots \times G(V_j)\rtarr G(f(V_1\oplus\cdots\oplus V_j))\]
is given by the $\sI$-FCP structure on $G$, and the map $S^f$ is the one-point compactification of $f$.
Here $G(V)$ acts on $Y$ through $F(V)$, and we need a compatibility condition relating this
action and the operad action for these maps to be well-defined, essentially compatibility with 
$d_0$ in the simplicial bar construction.  This then gives $T(G;Y)$ a structure of 
$\sL$-prespectrum.  We spectrify to obtain a Thom $\sL$-spectrum $M(G;Y)$, which is thus
an $E_{\infty}$ ring spectrum.  

The compatibility condition holds for $Y = GL_1R$ for an $E_{\infty}$
ring spectrum $R$ since the action comes via composition from the inclusion 
$GL_1R\subset R_0\iso\OM^V R(V)$.  This is the $Y$ that was considered in \cite[IV.33]{MQR}.
As on the classifying space level, we abbreviate notations by writing
\[  M(SG;SL_1R) = M(SG;R) \ \ \ \ \text{and} \ \ \ \  M(G;GL_1R) = M(G;R).\] 
The groups $\pi_*(M(G;R))$ are the cobordism groups of $G$-manifolds with 
$R$-oriented stable normal bundles when $G$ maps to $O$.  There is a similar 
interpretation using normal spaces in the sense of Quinn \cite{Quinn} when $G=F$. 
  
There is another perspective, which is suggested by Proposition \ref{twoside1}.
The $\sL$-spaces $Y$ of interest are often themselves $Y(\bR^{\infty})$ 
for a based $\sI_c$-FCP $Y$ with a right action by the $\sI_c$-FCP $F$ and 
therefore by the $\sI_c$-FCP $G$.  Using an alternative notation to make the 
distinction clear, we can then use the spaces $Y(V)$, $V$ finite dimensional, 
to form the Thom spaces
\[ T(Y,G,S)(V) = B(Y(V),G(V),S^V)/B(Y(V),G(V),\infty). \]
The $\sI$-FCP structures on $Y$ and $G$ and the $\sI$-FSP structure on $S$
give rise to maps 
\[ T(Y,G,S)(V)\sma T(Y,G,S)(W)\rtarr  T(Y,G,S)(V\oplus W),  \]
and there are evident maps $S^V\rtarr T(Y,G,S)(V)$. 
These give $T(Y,G,S)$ a structure of $\sI$-FSP.  As recalled from \cite[IV.2.5]{MQR}
in \cite[\S5]{Prequel1}, in analogy with the passage from $\sI$-FCP's to $\sL$-spaces, 
there is an easily defined functor from $\sI$-FSP's to $\sL$-prespectra that allows us 
to regard $T(Y,G,S)$ as an  $\sL$-prespectrum.  We let $M(Y,G,S)$ denote its 
spectrification, which is again an $\sL$-spectrum.  Up to language, this is the 
construction given in \cite[IV.2.5]{MQR}, and when both constructions apply they 
agree by \cite[IV.3.5]{MQR}.\footnote{There is a technical caveat here that we shall ignore.
The early argument just summarized required $\sI$-FSP's to satisfy an inclusion condition
to ensure that the relevant colimits are well-behaved homotopically. Arguments in \cite[I\S7]{MM}
circumvent that.} With this discussion in mind, it is now safe and convenient to consolidate 
notation by also writing $M(G;Y) = M(Y,G,S)$ when working from our second perspective.

\begin{rem}
The observant reader may wonder if we could replace
$S$ by another $\sI$-FSP $Q$ in the construction of $M(Y,G,-)$ and so get a 
generalized kind of Thom spectrum. We refer the reader to \cite[Ch. 23]{MS} 
for a discussion. Recent work gives explicit calculations of interesting 
examples \cite{MayTalk}.\footnote{Parenthetically, a quite different kind of generalized 
Thom spectrum is studied in \cite{five}.  Its starting point is to think of a delooping
$BGL_1R$ of $GL_1R$ as a ``classifying space'' with its own associated Thom spectrum $MGL_1R$,
analogous to and with a mapping from $MF = MGL_1S$.}
\end{rem}

The reader may also wonder if our second perspective applies to the construction
of $M(G;R)$, that is, if the $\sL$-space $GL_1R$ comes from an $\sI$-FCP $Y$.  To  answer that, 
we interject a more modern view of these two perspectives, jumping forward more than two decades to the 
introduction of orthogonal spectra.  When 
\cite{MQR} was written, it seemed unimaginable to me that all $E_{\infty}$ ring spectra arose 
from $\sI$-FSP's or that all $\sL$-spaces arose from $\sI$-FCP's.  I thought the second perspective 
given above only applied in rather special situations.  We now understand things better.  

As observed in \cite[\S\S2,12]{Prequel1}, $\sI$-FSP's are the external equivalent 
of commutative orthogonal ring spectra.  There are two different functors, equivalent 
up to homotopy, that pass from orthogonal
spectra to $S$-modules in the sense of \cite{EKMM}.  The comparison is made in 
\cite[Ch 1]{MM}.  The first functor is called $\bN$ and is the left adjoint of a 
Quillen equivalence.  It is symmetric monoidal and so takes commutative orthogonal 
ring spectra to commutative $S$-algebras, which, as we explained in \cite[\S11]{Prequel1}
are essentially the same as $E_{\infty}$ ring spectra.  The second is called $\bM$, and its
specialization to commutative orthogonal ring spectra is essentially the same functor from 
$\sI$-FSP's to $\sL$-prespectra to $\sL$-spectra of \cite{MQR} and \cite[\S5]{Prequel1} that we have been 
using so far in this section.\footnote{This holds when the inclusion condition we are ignoring holds
on the given $\sI$-FSP's.}  The conclusion is that, from the point of view of stable homotopy theory, 
we may use $\sI$-FSP's and $E_{\infty}$ ring spectra interchangeably, although $\sI$-FSP's 
do not directly encode $E_{\infty}$ ring spaces.

There is an analogous comparison of $\sI$-FCP's and $\sL$-spaces, 
although this has not yet appeared in print.\footnote{It starts from 
material in Blumberg's thesis \cite{Blum} and has been worked out in 
detail by Lind.}  In essence, it mimics the spectrum level 
constructions of \cite{EKMM, MM} on the space level.  Via that 
theory, we can also use $\sI$-FCP's and $\sL$-spaces interchangeably.  

This suggests that, when $R = \bM P$ for an $\sI$-FSP $P$, we can reconstruct $M(G;R)$ 
from the second perspective by taking $Y$ to be an explicitly defined $\sI$-FCP $GL_1P$.  
Using unit spaces $GL_1(P)(V) \subset \OM^V P(V)$, the required definition of $GL_1P$ is 
given in \cite[23.3.6]{MS}.  There is a subtle caveat in that $P$ must be fibrant in the 
``positive stable model structure'', so that $P(0) = S^0$ and $P$ behaves otherwise as an 
$\OM$-prespectrum.  Then the maps $P(V)\sma P(W)\rtarr P(V\oplus W)$ of the given FSP
structure on $P$ induce maps
\[ \OM^V P(V)\times \OM^W P(W)\rtarr \OM^{V\oplus W}P(V\oplus W) \] 
that specify a natural transformation of functors on $\sI\times \sI$ and restrict to maps
\[ GL_1P(V)\times GL_1P(W)\rtarr GL_1P(V\oplus W). \] 
These maps give $GL_1P$ the required structure of an $\sI$-FCP.\footnote{We are again ignoring
inclusion conditions; Lind's work shows how to get around this.}  

\begin{rem}  Assuming or arranging the inclusion condition in the definition of an $\sI$-FCP,
we can extend the functor $GL_1P$ to $\sI_c$ by passage to colimits.  This gives an $\sI_c$-FCP 
$GL_1P$ with a right action by any monoid-valued $\sI_c$-FCP $G$ over $F$ and thus places us in 
the context to which Proposition \ref{twoside1} can be applied to construct an $\sI_c$-FCP 
$B(GL_1P,G,*)$. With $R=\bM P$, the associated $\sL$-space is homeomorphic to the $\sL$-space 
$B(G;R) = B(R_0,G,*)$ obtained by application of Proposition \ref{BYGXinf} in \S2. 
\end{rem}

\section{Thom spectra and orientation theory}\label{ThomThom}

We interject some useful general results about the Thom spectra $M(G;Y)$, following 
\cite[IV.2.7, IV.2.8, IV.3.4, IV.3.5]{MQR}.\footnote{These are all labeled
``Remarks''. Frank Quinn once complained to me that some of our most interesting results 
in \cite{MQR} were hidden in the remarks.  He had a point.}  We then return to orientation 
theory and put them to use to compare universal orientations on the space level and on the
spectrum level.  

We observe first that the generic maps 
$q\colon B(Y,G,X)\rtarr B(*,G,X)$ induce corresponding maps of Thom spectra.

\begin{lem}  For an $\sL$-space $Y$ with a compatible right action by $F$ or for an 
$\sI$-FCP $Y$ with a right action of the monoid-valued $\sI$-FCP $F$, there is a 
canonical map of $\sL$-spectra $q\colon M(G;Y)\rtarr MG$. 
\end{lem}

The generic maps $\epz\colon B(Y,G,X)\rtarr Y\times_G X$ also induce certain corresponding
maps of Thom spectra.  For $H\rtarr G$, we can take $Y =H\backslash G = B(*,H,G)$ to obtain
\[  \epz\colon B(H\backslash G,G,X) \rtarr B(*,H,G)\times_{G}X \iso B(*,H,X). \]
Applied to a map $H\rtarr G$ of monoid-valued $\sI$-FCP's and the $\sI$-FSP $X=S$, this
induces a map of Thom spectra.  

\begin{lem}  There is a canonical map of $\sL$-spectra 
$\epz\colon M(G; H\backslash G) \rtarr MH$.
\end{lem}

The homotopy groups of $M(G; H\backslash G)$ are the cobordism classes of $G$-manifolds
with a reduction of their structural group to $H$.  We have a related map given by
functoriality in the variable $Y$, applied to $e=e\com j\colon G\rtarr GL_1R$. 

\begin{lem} There is a canonical map of $\sL$-spectra
$Me\colon M(H;G) \rtarr M(G;R)$. 
\end{lem} 

A less obvious map is the key to our understanding of orientation theory.
We use our first construction of $M(G;R)$ for definiteness and later arguments.
Since $R_0 \iso \OM^V R(V)$, 
we may view $GL_1R$ as a subspace of $\OM^V R(V)$. The evaluation map 
$\epz\colon GL_1R\times S^V\rtarr R(V)$ factors through the orbits
under the action of $G(V)$, and we may compose it with
$\epz\colon B(GL_1R,G(V),S^V)\rtarr GL_1R(V)\times_{G(V)}S^V$
to obtain a map $\xi\colon B(GL_1R,G(V),S^V)\rtarr R(V)$.  This works equally
well if we start with $R = \bM P$. We again get an induced map of $\sL$-spectra.

\begin{lem}\label{defMorient} There is a canonical map of $\sL$-spectra 
$\xi\colon M(G;R) \rtarr R$.
\end{lem}

Taking $R= MG$ and recalling the map $j\colon H\backslash G\rtarr GL_1(MH)$ of Proposition \ref{bored},
we obtain the following analogue of that result. 

\begin{prop}  The following diagram of $\sL$-spectra commutes.
\[ \xymatrix{
MH & M(G;H\backslash G) \ar[l]_-{\epz} \ar[r]^-{q} \ar[d]^{Mj}  & MG \\
& M(G;MH) \ar[ul]^{\xi} \ar[ur]_{q} & \\} \]
\end{prop}

Specializing to $H=G$ and again recalling that $G\backslash G = EG$, this gives the following
analogue of Corollary \ref{BTaut}.

\begin{cor}\label{ETaut} The following diagram of $\sL$-spectra commutes, its map
$\epz$ and top map $q$ are equivalences, and $q\com \epz^{-1}= \text{id}$ in the homotopy 
category
of spectra.
\[ \xymatrix{
MG \ar@{=}[d] & M(G;EG) \ar[l]_-{\epz} \ar[r]^-{q} \ar[d]^{Mj}  & MG \ar@{=}[d] \\
MG & M(G;MG) \ar[l]^-{\xi} \ar[r]_-{q} & MG\\} \]
Therefore $MG$ is a retract and thus a wedge summand of $M(G;MG)$ such that the map
$\xi$ and the lower map $q$ both restrict to the identity on the summand $MG$.
\end{cor}

We now reconsider $R$-orientations of $G$ from the Thom spectrum perspective.
There is an obvious quick definition, but on first sight it is not obvious how 
it relates to the definition that we gave on the classifying space level.

\begin{defn} An $R$-orientation of $G$ is a map of ring spectra $\mu\colon MG\rtarr R$;
it is an $E_{\infty}$ $R$-orientation if $\mu$ is a map of $E_{\infty}$ ring spectra.
\end{defn}

An $E_{\infty}$ $R$-orientation $\mu$ induces a map of 
$\sL$-spaces $\mu\colon GL_1(MG)\rtarr GL_1R$ and therefore a map of 
$\sL$-spectra $M\mu\colon M(G;MG)\rtarr M(G;R)$. 
We can glue the diagram of the following result to the bottom of the diagram of Corollary \ref{ETaut}.

\begin{prop}\label{propMorient} 
Let $\mu\colon MG\rtarr R$ be a map of $\sL$-spectra. Then the following diagram
of $\sL$-spectra commutes.
\[ \xymatrix{
MG \ar[d]_{\mu} & M(G;MG) \ar[l]_-{\xi} \ar[r]^-{q} \ar[d]_{M\mu} & MG \ar@{=}[d] \\
R & M(G;R) \ar[l]^-{\xi} \ar[r]_-{q} & MG\\} \]
Therefore $MG$ is a retract and thus a wedge summand of $M(G;R)$, and the lower map
$\xi$ restricts to the given map $\mu$ on this wedge summand.
\end{prop}

Given a map of ring spectra $\mu\colon MG\rtarr R$, the composite 
\[ \xymatrix@1{
MG \ar[r]^-{\epz^{-1}} & M(G;EG) \ar[r]^-{Mj} \ar[r] &  M(G;MG) \ar[r]^-{M\mu} & M(G;R)\\} \]
in the homotopy category of spectra is the Thom spectrum analogue of the map $g\colon BG\rtarr B(G;R)$ 
in our original definition of an $R$-orientation of $G$.  If $\mu$ is an $E_{\infty}$ ring 
map, then it induces a map of $\sL$-spaces $B\mu\colon B(G;MG)\rtarr B(G;R)$, and we can use 
Corollary \ref{BTaut} to obtain the following diagram of $E_{\infty}$ maps. 
\[ \xymatrix@1{
BG & B(EG,G,*) \ar[l]_-{\epz}^-{\htp} \ar[r]^-{Bj} & B(G;MG) \ar[r]^{B\mu} & B(G;R)
\\} \]
Since $\epz$ is an equivalence, this gives us an $E_{\infty}$ orientation $g$.

Conversely, given an $E_{\infty}$ orientation $g\colon BG\rtarr B(G;R)$, we can ``Thomify'' it to 
an $E_{\infty}$ ring map $Mg\colon MG\rtarr M(G;R)$ and then compose $Mg$ 
with $\xi\colon M(G;R)\rtarr R$ to get an $E_{\infty}$ $R$-orientation $\mu\colon MG\rtarr R$.
The required Thomification can be obtained by applying the methods of Lewis \cite[Ch IX]{LMS} 
to pull back the Thom spectrum $M(G;R)$ along the map $g$ to obtain a Thom 
spectrum $g^*M(G;R)$. The $E_{\infty}$ ring spectrum $M(G;R)$ is equivalent to $q^*MG$, 
$q\colon B(G;R)\rtarr BG$, and the $E_{\infty}$ homotopy $q\com g\htp \id$ implies that the 
composite $g^*M(G;R)\rtarr M(G;R)\rtarr MG$ is an equivalence of $E_{\infty}$ ring 
spectra.\footnote{This is cryptic since the best way to carry out the details uses 
parametrized spectra \cite{MayTalk, MS} and full details of how this and other such 
arguments should go have not yet been written up.} The
Thomification $Mg$ is the composite of the inverse of this equivalence and the
canonical map $g^*M(G;R)\rtarr M(G;R)$.

Technical details are needed to check that these constructions are mutually inverse, but the idea
should be clear.  In any case, with our present state of knowledge, we understand how to prove things on the 
spectrum level much better than on the space level, and it is usually easiest to construct spectrum 
level $E_{\infty}$ orientations and then deduce space level $E_{\infty}$ orientations, as in Theorem \ref{Joach}.

\section{Examples of bipermutative categories}\label{exmps}

Before turning to bipermutative categories, consider a topological rig (semi-ring) $A$.
It is an $(\sN,\sN)$-space, and so can be viewed as an $E_{\infty}$ ring space.  By
the multiplicative black box of the first prequel \cite{Prequel1}, it has an associated $E_{\infty}$ ring
spectrum $\bE A$ and an associated ring completion $\et\colon A\rtarr \bE_0A$.  On $\pi_0$,
this constructs the ring associated to $\pi_0(A)$ by adjoining negatives, hence it is an
isomorphism if $A$ is already a ring.   When $A$ is discrete, $H_i(\bE_0A) = 0$ for $i>0$ 
and $\et$ is a homotopy equivalence.  Therefore $\bE A$ is an Eilenberg--Mac\,Lane
spectrum $HA$, and this gives $HA$ an $E_{\infty}$ ring structure.

Now let $\sA$ be a bipermutative category.  We agree to write $B\sA$ for the $E_{\infty}$
ring space equivalent to the usual classifying space that we obtain by the constructions 
developed in the prequel \cite{Prequel2}.  The multiplicative black box of the first
prequel \cite{Prequel1} gives
an $E_{\infty}$ ring spectrum $\bE B\sA$ and a ring completion $\et\colon B\sA\rtarr \bE_0\sA$. 
Up to inverting a map that is an $E_{\infty}$ ring map and an equivalence, $\et$ is a map of 
$E_{\infty}$ ring spaces, where we understand $E_{\infty}$ ring spaces to mean $(\sC,\sL)$-spaces.  
Of course, we require $0\neq 1$ in $\sA$. As in \cite[\S10]{Prequel1}, we agree to write 
$\GA B\sA = \bE_0 B\sA$ and then to use notations like $\GA_nB\sA$ to denote components of 
this space.  Changing back from \cite{Prequel2}, we use the standard notation 
for monads that we used in \cite{Prequel1}, so that $CX$ denotes the usual $\sC$-space with a group completion $\al\colon CX\rtarr QX$, and similarly for other operads.  
More details of the following discussion are in \cite[VI\S5, VII\S1]{MQR}. 

An important first example of a bipermutative category is the free bipermutative category 
$\sE$ generated by its unit elements $\{0,1\}$.  It is the sub bipermutative category 
of isomorphisms in $\sF$.  Its rig of objects is the rig $\bZ_+$ of non-negative integers.  
There are no morphisms $m\rtarr n$ for $m\neq n$, and $\sE(n,n)$ is the symmetric group $\SI_n$.  
The sum $\SI_m\times \SI_n\rtarr \SI_{m+n}$ is obtained by ordering the set of $m+n$ 
objects as the set of $m$ objects followed by the set of $n$ objects.  The product
$\SI_m\times \SI_n\rtarr \SI_{mn}$ is obtained by lexicographically ordering
the set of $mn$ objects.  The commutativity isomorphisms are the evident 
ones \cite[VI.5.1]{MQR}.  
There is a unique map $e\colon \sE\rtarr \sA$ of 
bipermutative categories from $\sE$ to any other bipermutative category $\sA$.

Since $CS^0$ is the free $(\sC,\sL)$-space generated by $S^0$, we have a unique 
$(\sC,\sL)$-map $\nu\colon CS^0\rtarr B\sE$.  Up to homotopy, both source and 
target are the disjoint union of classifying spaces $B\SI_n$, $n\geq 0$, and  
$\nu$ is an equivalence.  As we recalled in \cite[10.1]{MQR}, one version of the 
Barratt--Quillen theorem says that $\GA CS^0\htp \GA B\sE$ is equivalent 
to $QS^0$ as an $E_{\infty}$ ring space.  For a bipermutative category $\sA$, the 
unit $e\colon \sE\rtarr\sA$ induces the unit map  
$e\colon CS^0\iso B\sE\rtarr B\sA$ of the $(\sC,\sL)$-space $B\sA$, which in turn induces
the unit map $e\colon S\rtarr \bE B\sA$ of the associated $\sL$-spectrum.

As in the case of $\sE$, all of the following examples of bipermutative categories $\sA$
have $\bZ_+$ as their rig of objects and have no
morphisms $m\rtarr n$ for $m\neq n$ and a group of morphisms $\sA(n,n)$.  Let $R$ be 
a commutative topological ring, such as $\bR$, $\bC$, or a discrete commutative
ring.  We then have a bipermutative category  $\sG\sL R$ whose $n^{th}$ group
is $GL(n,R)$.  The sum and product are given by block sum and tensor product
of matrices, where the latter is interpreted via lexicographic ordering of 
the standard basis elements of $R^{mn}$.  

\begin{exmp}\label{Ex1} When $R$ is $\bR$ or $\bC$, we can restrict to orthogonal or
unitary matrices without changing the homotopy type, and we write $\sO$
and $\sU$ for the resulting bipermutative categories.  Then $\bE B\sO$ and 
$\bE B\sU$ are models of the connective $K$-theory spectra $kO$ and $kU$ 
as $E_{\infty}$ ring spectra, by \cite[VII.2.1]{MQR}. 
\end{exmp}  

\begin{exmp}\label{Ex2} When $R$ is discrete, we write $KR = \bE B\sG\sL R$.  It
is a model for (connective) algebraic $K$-theory, as defined by Quillen;
That is, $\pi_i(KR)$ is Quillen's $i^{th}$ algebraic $K$-group $K_i(R)$ for $i>0$.  
See \cite[VII\S1]{MQR}. Here $\pi_0(KR) = \bZ$.  We can obtain the correct $K_0R$ 
without changing the higher homotopy groups by replacing $\sG\sL R$ by a skeleton of the
symmetric bimonoidal category of finitely generated projective $R$-modules
and their isomorphisms.
\end{exmp}

Recall from \cite[\S7]{Prequel1} that $C(X_+)$ is a $(\sC,\sL)$-space if $X$ is an $\sL$-space
and $\bE C(X_+)$ is equivalent to the $\sL$-spectrum $\SI^{\infty}_+(X)\equiv \SI^{\infty}(X_+)$
with $0^{th}$ space $Q(X_+)$.

\begin{exmp}\label{Ex3}  Define $\sO R\subset \sG\sL R$ to be the subbipermutative
category of orthogonal matrices, $MM^t=\Id$, and write $KO R = \bE B\sO R$.   
This example is sometimes interesting and sometimes not.  For instance, $O(n,\bZ)$ is isomorphic 
to the wreath product $\SI_n\int \pi$, where $\pi$ is cyclic of order $2$, and there is an equivalence 
$C(B\pi_+)\rtarr B\sO(\bZ)$. This implies that $\bE B\sO \bZ$ is equivalent as an 
$E_{\infty}$ ring spectrum to $\SI^{\infty}_+B\pi$.  See \cite[VI.5.9]{MQR}.  
Variants of the $\sO R$ are often of interest.
\end{exmp}

The remaining examples here will be applied to topology in the next two sections. 
The importance of the following construction will become clear in \S\ref{last}.

\begin{exmp}\label{Ex4}  Let $X$ be a $\sD$-space for any $E_{\infty}$ operad $\sD$.  
For $r\in\pi_0(X)$, define $e_r\colon S^0\rtarr X$ by sending $0$ 
to the operadic basepoint of $X$ and sending $1$ to any chosen basepoint in 
the $r^{th}$ component.  The composite
of $De_r\colon DS^0\rtarr DX$ and the action $DX\rtarr X$ specifies a map of
$\sD$-spaces $DS^0\rtarr X$.  It is called an exponential unit map of $X$ and,
up to homotopy of $\sD$-maps, it is independent of the choice of basepoint.
\end{exmp}

\begin{exmp}\label{Ex5} Let $r=q^a$, $q$ prime. Let $\bF_r$ be the field with $r$ elements 
and let $\bar{\bF}_q$ be its algebraic closure.  Let $\ph^q$ denote the Frobenius
automorphism of $\sG\sL \bar{\bF}_q$, which raises matrix entries to the $q^{th}$
power, and let $\ph^r$ denote its $a$-fold iterate.  Then $\ph^r$ is an 
automorphism of bipermutative categories that restricts to an automorphism of
$\sO\bar{\bF}_q$. Moreover, the fixed point bipermutative category of $\ph^r$ is 
$\sG\sL\bF_r$. 
\end{exmp}

\begin{exmp}\label{Ex6} Again, let $r = q^a$. Define the forgetful functor 
$f\colon \sG\sL \bF_r\rtarr \sE$ as follows. On objects, let $f(n) = r^n$.  
Fix an ordering of the underlying set of $\bF_r$ and order $\bF^n_r$ 
lexicographically.  Then regard a matrix $M\in GL(n,\bF_r)$ as a permutation
of the ordered set $\bF^n_r$.  The functor $f$ is an exponential map of
permutative categories $(\sG\sL\bF_r,\oplus) \rtarr (\sE,\otimes)$. As we
recalled in \cite[9.6]{Prequel1} and \cite[\S11]{Prequel2}, the Barratt--Eccles operad 
$\sD$ acts on the classifying space of any permutative category.  The composite 
map of $\sD$-spaces
\[ \xymatrix@1{
DS^0 \iso B(\sE,\oplus) \ar[r]^{Be} & B(\sG\sL \bF_r,\oplus) \ar[r]^{Bf} & B(\sE,\otimes)\\} \]
coincides with the exponential unit $e_r$ of Example \ref{Ex4}. This works equally well with
$\sG\sL$ replaced by $\sO$. 
\end{exmp}

\begin{exmp}\label{Ex7}  Let $r=3$.  Then the subcategory $\sN \bF_3\subset\sO \bF_3$ 
of matrices $M$ such that $\nu(M)\det(M) = 1$ is a subbipermutative
category, where $\nu$ is the spinor norm.  See \cite[VI.5.7]{MQR}.
\end{exmp}

\section{Brauer lifting on the infinite loop space level}\label{Brauer}

For simplicity and definiteness, we fix a prime $p$ and complete all spaces and spectra at $p$ 
throughout this section.\footnote{Advertisement: Kate Ponto and I are nearing completion of a sequel 
to ``A concise course in algebraic topology'' which will give
an elementary treatment of localizations and completions.}  We let $r = q^a$ for 
some other prime $q$. 

The group completion property of the additive infinite loop space machine 
implies that the map $\et\colon B\sA\rtarr \bE_0B\sA\equiv \GA B\sA$ induces a homology isomorphism 
\[ \bar{\eta}\colon BA_{\infty}\rtarr \GA_0 B\sA \] for any of the categories $\sA$ displayed
in the previous section, where $A_{\infty}$ is the colimit of the groups $\sA(n,n)$. 
For example, $H_*(BGL(\infty,R)) \iso H_*(\GA_0 B\sG\sL R)$ for a (discrete) 
commutative ring $R$.  

Quillen's proof of the Adams conjecture \cite{Quillen}, which is what led him to the definition 
and first computations in algebraic $K$-theory, was based on Brauer lifting of representations
in $GL(n,\bar{F}_q)$ to (virtual) complex representations.  He did not yet have completion
available, and so the calculations were mysterious, producing a mod $p$ homology isomorphism
from a space with homotopy groups in odd degrees, the algebraic $K$-groups $K_i(\bar{F}_q)$, 
to a space with homotopy groups in even degrees, the topological $K$-groups $K_i(S^0)$.  

Completion explained the mystery.  While completion was available when \cite{MQR} was written, 
it was not yet known that completions of $E_{\infty}$ ring spectra are $E_{\infty}$ ring spectra.  
In fact, that was not proven until \cite{EKMM}.  While this fact allows a slightly smoother
exposition of what follows than was given in \cite[Ch. VII]{MQR}, the improvement is small.
Since that chapter is less affected by later developments than most others in
\cite{MQR} and should still be readable, we shall just summarize the main lines of
argument.  

The idea of \cite[Ch. VIII]{MQR} is to apply constructions in algebraic $K$-theory to 
gain information in geometric topology by using algebraic $K$-theory to construct
``discrete models'' for spaces and spectra of geometric interest, thus showing that 
they have more structure than we would know how to derive working solely from a 
topological perspective.  When given some space or spectrum $X$ of geometric interest, 
we write $X^{\de}$ for such a discrete approximation.  

The essential point is to analyze Brauer lifting on the $E_{\infty}$ level.  
As proven by Quillen \cite{Quillen} and summarized in \cite[VIII\S2]{MQR}, 
after completing at any prime $p\neq q$, Brauer lifting of representations 
leads to equivalences
\begin{equation}\label{Brspace} \la\colon BU^{\de} \equiv \GA_0B\sG\sL \bar{F}_q\rtarr BU
\ \ \ \text{and} \ \ \ \la\colon BO^{\de} \equiv \GA_0B\sO \bar{F}_q\rtarr BO.  
\end{equation}
Here we are thinking a priori just about homotopy types, despite the $\GA_0$ notation.
We use the same notation when thinking of the $H$-space structure induced by $\oplus$,
but we add a subscript $\otimes$ when thinking about the $H$-space structure induced
by $\otimes$.   By representation theoretic arguments, it is shown that
the maps $\la$ are equivalences of $H$-spaces under either $H$-space structure \cite[VIII.2.4]{MQR}
and that they convert the Frobenius automorphism $\ph^r$ to the Adams operation $\ps^r$, meaning
that $\ps^r\com \la\htp \la\com \ph^r$ \cite[VIII.2.5]{MQR}.  

The fact that $\la$ is an $H$-map under $\otimes$ implies a compatibility statement with respect 
to multiplication by the Bott class.  Using an elementary and amusing equivalence between 
``periodic connective spectra'' and  periodic spectra \cite[pp. 43--48]{MQR}, this leads
to a proof that the maps $\la$ of (\ref{Brspace}) are the maps on the $0^{th}$ component 
of the $0^{th}$ space of equivalences 
\begin{equation}\label{Brspectrum}
 \la\colon kU^{\de} \equiv \bE B\sG\sL \bar{F}_q\rtarr kU
\ \ \ \text{and} \ \ \ \la\colon kO^{\de} \equiv \bE B\sO\bar{F}_q\rtarr kO.
\end{equation}
of ring spectra up to homotopy \cite[VIII.2.8]{MQR}.  Moreover, these spectrum level equivalences are
uniquely determined by the space level equivalences $\la$, and we have $\ps^r\com \la\htp \la\com \ph^r$ 
on the spectrum level \cite[VIII.2.9]{MQR}. 

All four spectra displayed in (\ref{Brspectrum}) are $E_{\infty}$ ring spectra.  One 
would like to say that the maps $\la$ are maps of $E_{\infty}$ ring spectra, the Adams
maps $\ps^r$ are maps of $E_{\infty}$ ring spectra, and $\ps^r\com \la\htp \la\com \ph^r$ 
as maps of $E_{\infty}$ ring spectra.  Conceivably these statements could be proven using 
modern techniques, although I have no idea how to do so, but proofs were unimaginable when 
\cite{MQR} was written.  Tornehave and I proved enough that we could calculate just as if these statements 
were true.  I'll sketch how we did this. 

Recall that $BU_{\otimes} = SL_1kU$ and $BO_{\otimes} =  SL_1kO$. 
Similarly, write $BU^{\de}_{\otimes} = SL_1kU^{\de}$ and $BO^{\de}_{\otimes} = SL_1kO^{\de}$. 
By passage to  $1$-components of $0^{th}$ spaces from the equivalences of (\ref{Brspectrum}),
we obtain equivalences of $H$-spaces:
\begin{equation}\label{Ude}
 \la_{\otimes}\colon BU^{\de}_{\otimes}= \GA_1B\sG\sL\bar{F}_q\rtarr \GA_1B\sU = BU_{\otimes},
\end{equation}
\begin{equation}\label{Ode}
\la_{\otimes}\colon BO^{\de}_{\otimes}= \GA_1B\sO\bar{F}_q\rtarr \GA_1B\sO = BO_{\otimes}.
\end{equation}
The understanding of localizations of $sl_1(R)$ for an $E_{\infty}$ ring spectrum $R$ that we 
described in \cite[\S 10]{Prequel1} comes into play in the proof of the following result, 
which is \cite[VII.2.11]{MQR}.  We give an outline sketch of its somewhat lengthy proof.

\begin{thm}\label{Brtensor}  The $H$-equivalences 
\[ \la_{\otimes}\colon BU^{\de}_{\otimes}\rtarr BU_{\otimes} \ \ \ \text{and} \ \ \
\la_{\otimes}\colon BO^{\de}_{\otimes}\rtarr BO_{\otimes}  \]
are equivalences of infinite loop spaces.
\end{thm}
\begin{proof}[Sketch proof]  It is easy to prove that we have splittings of infinite 
loop spaces 
\[
BU_{\otimes}\htp BU(1)\times BSU_{\otimes} \ \ \ \text{and} \ \ \
BO_{\otimes}\htp BO(1)\times BSO_{\otimes} 
\]
\[BU^{\de}_{\otimes}\htp BU(1)\times BSU^{\de}_{\otimes} \ \ \ \text{and} \ \ \
BO^{\de}_{\otimes}\htp BO(1)\times BSO^{\de}_{\otimes};
\]
see \cite[V.3.1, VII.2.10]{MQR}. Here $BSU^{\de}_{\otimes}$ is the $3$-connected cover of $BU^{\de}_{\otimes}$ and $BSO^{\delta}_{\otimes}$ 
is the $1$-connected cover of $BO^{\delta}_{\otimes}$.  Thinking topologically, the idea is to think of 
$BU(1)\htp K(\bZ,2)$ and $BO(1)\htp K(\bZ/2,1)$  as representing the functors giving Picard groups 
of complex or real line bundles, but the proof is homotopical. The equivalences $\la_{\otimes}$ respect the splittings, and the resulting $H$-equivalences of Eilenberg--Mac\,Lane spaces 
are clearly equivalences of infinite loop spaces.  Thus it suffices to prove the result with $U$ and $O$ replaced 
by $SU$ and $SO$ in the statement.

By a result of Adams and Priddy \cite{AP}, $BSU$ and $BSO$ have unique infinite loop structures.
By a result of Madsen, Snaith, and Tornehave \cite{MST}, if $X$ and $Y$ are infinite loop spaces both 
homotopy equivalent to $BSU$ or to $BSO$, then an $H$-map $f\colon X\rtarr Y$ is an infinite loop 
map if and only if it ``commutes with transfers'' or, in the language of \cite[VIII.1.3, 1.5]{MQR}, 
is an ``$H_{\infty}^p$--map''.  Letting $W \htp \sC(p)$, the infinite loop structure gives the maps $\tha$ 
in the following diagram, and $f$ is an $H_{\infty}^p$--map if the diagram is homotopy commutative.
Here $\pi$ is the cyclic group of order $p$.
\[ \xymatrix{
W\times_{\pi} X^p \ar[r]^-{\text{id}\times f^p} \ar[d]_{\tha} & W\times_{\pi} Y^p \ar[d]^{\tha} \\
X \ar[r]_-{f}  & Y \\} \]

From here on, the argument is the same in the two cases and we focus on the complex case.  
The spaces of (\ref{Ude}) are constructed from the infinite loop space machine,
viewed as a multiplicative enrichment of the additive infinite loop space machine.
Let $M\subset \bZ_+$ be the monoid of integers prime to $p$.
We have permutative categories $(\amalg_{m\in M} GL(m,\bar{F}_q),\otimes)$ and
$(\amalg_{m\in M}U(m),\otimes)$.  Let $X$ and $Y$ denote their classifying spaces.  
We can apply the infinite loop space machine to $X$ and $Y$ to obtain $\GA X$ and $\GA Y$, 
and we have the group completions $\et\colon X\rtarr \GA X$ and $\et\colon Y\rtarr \GA Y$.
By \cite[10.1]{Prequel1}, we have equivalences of infinite loop spaces 
\[ \io\colon \GA_1X \rtarr  \GA_1 B\sG\sL\bar{F}_q \ \ \ \text{and} \ \ \  
\io\colon \GA_1Y \rtarr  \GA_1 B\sU.  \] 
It suffices to prove that $\la_{\otimes} \com \io$ is an infinite loop map.  For that, it suffices to show 
that $\la_{\otimes}\com \io$ is an $H_{\infty}^p$-map since its restriction to 
$3$-connected covers will then also be an $H_{\infty}^p$-map.  The equivalences $\io$ 
extend over components to equivalences 
\[ \io\colon \GA X \rtarr  \GA B\sG\sL\bar{F}_q \ \ \ \text{and} \ \ \  
\io\colon \GA Y \rtarr  \GA B\sU.  \] 
By interpreting the restriction of the relevant maps $\tha$ along the group completion 
maps $\et$ and chasing a fairly elaborate but elementary diagram involving change of components
\cite[VIII.1.2, 1.4]{MQR}, we find that it suffices to prove that the following diagram is 
homotopy commutative for each $m\in M$.
\[ \xymatrix{
W\times_{\pi} BGL(m,\bar{F}_q)^p\ar[r]^-{\htp} \ar[d]_{\text{id}\times \be^p}
& B(\pi\int GL(m,\bar{F}_q))  \ar[r]^-{B\tilde{c}_{\otimes}} 
& BGL(m^p,\bar{F}_q) \ar[d]^{\be}\\
W\times_{\pi} (BU\times\{m\})^p \ar[r]_-{\htp} 
& B(\pi\int U^p)\times\{m^p\} \ar[r]_-{B\tilde{c}_{\otimes}} & BU\times\{m^p\}.\\} \]
Here the homomorphisms $\tilde{c}_{\otimes}$ are induced by the tensor product and commutativity isomorphisms
of our two permutative categories.   The maps $\be$ are given by Brauer lifting of representations,
and the argument so far reduces the question to an algebraic problem in representation theory.  Its 
solution requires careful use of various standard results from Serre \cite{Serre} that allow us to lift
relevant representations in finite fields to honest rather than virtual complex representations.  The map
$\be$ involves a choice of embedding $\mu\colon \bar{\bF}_q^{\times} \rtarr \bC^{\times}$ of roots of unity 
in the complex numbers, and the proof depends on making a particularly good choice, consistent with a certain
decomposition isomorphism; details are in \cite[pp. 220--222]{MQR}. 
\end{proof}

\section{The $K$-theory of finite fields and orientation theory}\label{last}

We return to the discussion of infinite loop space theory and orientation theory that 
we started in \S\ref{units}.  We describe some of the results that provided the
original motivation for the theory of $E_{\infty}$ ring spaces and $E_{\infty}$ ring
spectra.  Much of the work of \cite{MQR} focused on three large diagrams
\cite[pp 107, 125, 229]{MQR}.\footnote{It would be nice to have these diagrams readably texed; I haven't
tried.  Another advertisement: \cite{MQR} and related early books have been scanned and are available online at
http://www.math.uchicago.edu/~may/BOOKSMaster.html.}  We will extract some of the 
conclusions about them, highlighting the role of $E_{\infty}$ ring theory.

Again completing all spaces and spectra at a fixed prime $p$, we now take $r=3$ if $p=2$ 
and we assume that $r = q^a$ reduces mod $p^2$ to a generator of the group of units of 
$\bZ/p^2$ if $p$ is odd.  We abbreviate notation by writing $BC = BCoker\, J$ and $C= \OM BC$. 
Since $BSpin \htp BSO\htp BO$ at $p>2$, the definition of $BC$ in \S\ref{units} 
can be restated by letting $BC$ be the fiber of $c(\ps^r)\colon B(SF;kO)\rtarr BSpin_{\otimes}$ 
at any prime $p$. Similarly, define $J$ to be the fiber of $\ps^r-1\colon BO\rtarr BSpin$ at $p$.
When $p=2$, this is the most convenient (for the present purposes) of the several choices that can be made.

The $J$-theory diagram of \cite[p. 107]{MQR} implies a slew of
splittings of spaces of geometric interest, such as
$SF\htp J\times C$ and $B(SF;kO)\htp BSpin \times BC$.  The initial applications 
of $E_{\infty}$ ring theory were aimed at proving (or disproving) that these are 
splittings of infinite loop spaces.  Since our calculational understanding of the 
spaces in question depends on their Dyer-Lashof operations, which are invariants 
of the infinite loop structure, this analysis is essential to calculations.

To start things off, observe that the theory of \S\ref{units} and the equivalence 
$\la$ of (\ref{Brspectrum}) directly give the following two equivalences 
of fibration sequences, in which $BC^{\de}$ is defined to be the fiber of $c(\phi^r)$. 

\[ \xymatrix{
SF \ar@{=}[d]  \ar[r]^-{e} & BO^{\de}_{\ten} \ar[d]^{\la_{\ten}}_{\simeq} \ar[r]^-t
& B(SF;kO^{\de}) \ar[d]^{B\la}_{\simeq} \ar[r]^-{q} & BSF \ar@{=}[d]  \\
SF \ar[r]_-{e} & BO_{\ten} \ar[r]_-t & B(SF;kO) \ar[r]_-{q} & BSF. \\}
\]

All spaces in this diagram are infinite loop spaces.  The left square turns out to be a 
commutative diagram of infinite loop spaces \cite[VIII.3.4]{MQR}.  Therefore, by standard 
arguments with 
fibration sequences of spectra, we can take $B\la$ to be an infinite loop map such that
the diagram is a commutative diagram of infinite loop spaces.  Of course, this would have 
been automatic if we knew that $\la$ were a map of $E_{\infty}$ ring spectra. 

\[ \xymatrix{
Spin_{\otimes}^{\de} \ar[r] \ar[d]_{\OM \la_{\otimes}}^{\htp} & BC^{\de} \ar[r] \ar[d]^{\mu}_{\htp}
&B(SF;kO^{\de}) \ar[r]^-{c(\ph^r)} \ar[d]^{B\la}_{\htp} & BSpin^{\de}_{\otimes} \ar[d]^{\la_{\otimes}}_{\htp}\\
Spin_{\otimes} \ar[r] & BC \ar[r] 
& B(SF;kO) \ar[r]^-{c(\ps^r)} & BSpin_{\otimes} \\} \]

In this diagram, we do not know that $\ps^r$ is an $E_{\infty}$ ring map, so
we do not know that $BC$ is an infinite loop space.  Since $\phi^r$ is an $E_{\infty}$
ring map, $c(\ph^r)$ is an infinite loop map and $BC^{\de}$ inherits an infinite loop
structure such that the top fibration is one of infinite loop spaces.  The equivalence
$\mu$ is any map such that the diagram commutes, and we may regard it as specifying
a structure of infinite loop space on $BC$.  This allows us to regard the bottom
fibration as one of infinite loop spaces. 

There is a more illuminating description of $BC^{\de}$ that comes from further discrete models.
On the spectrum level, define $bo$, $bso$, and $bspin$ to be the covers of $kO$ with 
$0^{th}$ spaces $BO$, $BSO$, and $BSpin$. These are all the same if $p>2$.  
Define $\kappa\colon j\rtarr ko$ to be the fiber of $\ps^3-1\colon ko\rtarr bspin$.  
Then the zero component of the $0^{th}$ space of $j$ is $J$.  These spaces and spectra all
have discrete models, as proven in \cite[VIII\S3]{MQR}.  The essential starting point is Quillen's 
work on the $K$-theory of finite fields \cite{Quill2}, which shows in particular that, at $p>2$, 
$J$ is equivalent to the fiber $J^{\de}$ of the map $\ph^r-1\colon BU^{\de} \rtarr BU^{\de}$.  
Work of Fiedorowicz and Priddy \cite{FP} also plays a role in the following result.
Recall Example \ref{Ex7}. 

\begin{defn} Define the following spaces and spectra.
\begin{enumerate}[(i)]
\item At $p=2$, $j^{\de} = \bE B\sN\bF_3$; at $p>2$, $j^{\de} = K\bF_r$.  These are $E_{\infty}$ ring spectra.
\item $J^{\de}_{\oplus}$ and $J^{\de}_{\otimes}$ are the $0$ and $1$ components of the $0^{th}$ space of $j^{\de}$.  These are additive and multiplicative infinite loop spaces. 
\end{enumerate}
\end{defn}

The Brauer lift $\la$ of (\ref{Brspectrum}) and comparison of $\ps^r-1$ and $\ph^r-1$ 
leads to the following result \cite[VIII.3.2]{MQR}, although some intermediate comparisons 
and some minor calculations are needed for the proof. 

\begin{thm}\label{Brnew}  There is an equivalence of spectra 
$\nu$ and a commutative diagram 
\[ \xymatrix{
j^{\de}\ar[r]^-{\ka^{\de}} \ar[d]_{\nu} & ko^{\de} \ar[d]^{\la} \\
j\ar[r]_-{\ka} & kO \\} \]
in which $\ka^{\de}$ is induced by a map of bipermutative categories and $\ph^r\com \ka^{\de} = \ka^{\de}$.
\end{thm}

The last statement implies that the restriction of $c(\ph^r)\colon B(SF;kO^{\de})\rtarr Spin_{\otimes}^{\de}$
to the space $B(SF;j^{\de})$ is the trivial infinite loop map. There results an infinite loop map 
$\xi^{\de}\colon B(SF;j^{\de})\rtarr BC^{\de}$.  Since the analogous map $\xi\colon B(SF;j)\rtarr BC$ is an
equivalence \cite[V.5.17]{MQR}, we can deduce that $\xi^{\de}$ is so too.

\begin{cor} The infinite loop map $\xi^{\de}\colon B(SF;j^{\de})\rtarr BC^{\de}$ is an equivalence.
\end{cor}

At this point, one can put together a braid of topologically defined fibrations of interest, together with
an equivalence from a corresponding braid of discrete models that makes the whole diagram one of infinite
loop spaces \cite[VIII.3.4]{MQR}.  The braid focuses attention on the fibration sequence 
\[ \xymatrix@1{
SF \ar[r]^-{e} & J^{\de}_{\otimes} \ar[r]^-{t} & B(SF;j^{\de})\ar[r]^-{q} & BSF.\\} \]
Ignoring infinite loop structures, one sees from the homotopical splitting of $SF$ that $t$ is
null homotopic.  At $p=2$, it is not even true that  $SF\htp J\times C$ as $H$-spaces \cite[II.12.2]{CLM}, 
but, at $p>2$, $SF\htp J\times C$ as infinite loop spaces, as we now explain.  Let $M\subset \bZ_+$
be the submonoid of integers $r^n$ and let $\sE_M = \amalg_{m\in M} \SI_m$.   We use the exponential
unit map $e_r = f\com e\colon (\sE,\oplus) \rtarr (\sE_M,\otimes)$ of permutative categories 
described in Examples \ref{Ex4} and \ref{Ex5}.  The forgetful functor $f$ induces an infinite
loop map $J^{\de} = \GA_0B\sG\sL\bF_r\rtarr \GA_1(B\sE_M,\otimes)$.  By another application of
\cite[10.1]{Prequel1}, the target is equivalent
(away from $r$ and therefore at $p$) to $\GA_1(B\sE,\oplus)\htp SF$.  
Let $\al^{\de}\colon J^{\de}\rtarr SF$ be the resulting infinite loop map.  We have the commutative diagram
\[  \xymatrix{
Q_0 S^0 \ar[r]^-{\al^{\de}\com e} \ar[d]_{e} & SF \ar[d]_{e}\\
J^{\de} \ar[r]_-{e\com \al^{\de}} \ar[ur]^{\al^{\de}} & J^{\de}_{\otimes}, \\} \]
where the left and right vertical arrows $e$ are the restrictions to the components of
$0$ and $1$ of the map on $0^{th}$ spaces of the unit map $e\colon S\rtarr j^{\de}$. 
This works equally well at $p=2$, but at odd primes $p$ a direct homological calculation
using Quillen's calculation of $H_*(J^{\de};\bF_p)$ and analysis of Dyer--Lashof
operations gives the following exponential equivalence \cite[VIII.4.1]{MQR}, as
promised in the introduction. 

\begin{thm} At $p>2$, the composite $J^{\de} \overto{\al^{\de}} SF \overto{e} J^{\de}_{\otimes}$
is an equivalence. 
\end{thm}

As observed in \cite[pp 240--241]{MQR} this implies the following result.

\begin{cor} At $p>2$, there are equivalences of infinite loop spaces
\[ J\times C\htp SF, \ \ \text \ BJ\times BC\htp BSF, \ \ \text{and}\ \ B(SF;kO)\htp BO_{\otimes} \times BC. \]
\end{cor}

The notation $\al^{\de}$ suggests that there should be a precursor $\al\colon J\rtarr SF$,
and indeed there is.  Such a map comes directly from the Adams conjecture, and, at $p>2$, it is an
infinite loop map \cite{Fried}.  Moreover, away from $2$, work of Sullivan \cite{Sull} gives a
spherical orientation of $STop$ that leads to an equivalence of fibrations of infinite loop spaces
\[ \xymatrix{
SF \ar[r]^-{t} \ar[d]_{\chi} & F/Top \ar[r]^-{q} \ar[d]^{f} & BSTop \ar[d]^{g} \ar[r]^-{Bj} & BSF \ar@{=}[d]\\
SF \ar[r]_-{e} & BO_{\otimes} \ar[r]_-{t}  & B(SF;kO) \ar[r]_-{q} & BSF. \\} \]
This reduces the calculation of mod $p$ characteristic classes for topological bundles, $p\neq 2$, to
calculation of $H^*(BC;\bF_p)$. This is accessible via Dyer-Lashof operations in homology,
as worked out in \cite[Part II]{CLM}.  The essential point is that, at an odd prime $p$, we can replace 
$kO$ and $BO_{\otimes}$ by discrete models, and that reduces the calculation to calculations in the
cohomology of finite groups.


\begin{thebibliography}{99}

\bibitem{Adams}
J.F. Adams.
Infinite loop spaces. 
Annals of Math Studies \# 90.
Princeton University Press. 1978.

\bibitem{JX1}
J.F. Adams. On the groups J(X). I. Topology 2(1963), 181--195.

\bibitem{JX2}
J.F. Adams. On the groups J(X). II. Topology 3(1965), 137--171.

\bibitem{JX3} J.F. Adams. On the groups J(X). III. Topology 3(1965), 193--222.

\bibitem{JX4} J.F. Adams. On the groups J(X). IV. Topology 5(1966), 21--71.

\bibitem{AP}
J.F. Adams and S.B. Priddy.
Uniqueness of $B{\rm SO}$.  
Math. Proc. Cambridge Philos. Soc. 80(1976), 475--509. 

\bibitem{five}
M. Ando, A.J. Blumberg, D.J. Gepner, M.J. Hopkins, and C. Rezk. 
Units of ring spectra and Thom spectra.
Available at http://front.math.ucdavis.edu/0810.4535.

\bibitem{AHS}
M. Ando, M.J. Hopkins, and N.P. Strickland.
The sigma orientation is an $H_{\infty}$ map.
Amer. J. Math. 126(2004), 247--334.

\bibitem{Blum}
A.J. Blumberg. 
Progress towards the calculation of the K-theory of Thom spectra. 
University of Chicago thesis, 2005. 

\bibitem{Blum1}
A.J. Blumberg.
THH of Thom spectra that are $E_\infty$ ring spectra.
Available at http://front.math.ucdavis.edu/0811.0803.

\bibitem{BCS}
A.J. Blumberg, R.L. Cohen, and C. Schlichtkrull.
Topological Hochschild homology of Thom spectra
and the free loop space.
Available at http://front.math.ucdavis.edu/0811.0553.

\bibitem{CLM}
F.~R. Cohen, T.~J. Lada, and J.P. May.  
The homology of iterated loop spaces.
Lecture Notes in Mathematics Vol. 533. 
Springer-Verlag 1976.

\bibitem{EKMM}
A. Elmendorf, I. Kriz, M.~A. Mandell, and J.P. May. 
Rings, modules, and algebras in stable homotopy theory. 
Amer. Math. Soc. Mathematical Surveys and Monographs Vol 47. 1997.

\bibitem{FP}
Z. Fiedorowicz and S. Priddy. 
Homology of classical groups over finite fields and their 
associated infinite loop spaces. 
Lecture Notes in Mathematics, 674. Springer, Berlin, 1978.

\bibitem{Fried}
E.M. Friedlander.
The infinite loop Adams conjecture via classification 
theorems for ${\mathcal F}$-spaces.
Math. Proc. Cambridge Philos. Soc. 87(1980), 109--150. 

\bibitem{Joa}
M. Joachim. 
Higher coherences for equivariant $K$-theory.  
Structured ring spectra, 87--114. 
London Math. Soc. Lecture Note Ser., 315. 
Cambridge Univ. Press, Cambridge, 2004.

\bibitem{KS}
R.C. Kirby and L.C. Siebenmann.
Foundational essays on topological manifolds, smoothings, and triangulations. 
With notes by John Milnor and Michael Atiyah. 
Annals of Mathematics Studies, No. 88. 
Princeton University Press, Princeton, N.J.; University of Tokyo Press, Tokyo, 1977. 

\bibitem{LMS}
L.~G. Lewis, J.P. May, and M. Steinberger (with contributions by 
J.~E. McClure).  
Equivariant stable homotopy theory.  
Lecture Notes in Mathematics Vol. 1213.  
Springer-Verlag, 1986.


\bibitem{MST}
I. Madsen, V.P. Snaith, and J. Tornehave. 
Infinite loop maps in geometric topology.  
Math. Proc. Cambridge Phil. Soc.  81(1977), 399--430.

\bibitem{MM}
M.~A. Mandell and J.P. May.
Equivariant orthogonal spectra and $S$-modules. 
Memoirs Amer. Math. Soc. Vol 159. 2002. 

\bibitem{Class}
J.P. May.
Classifying spaces and fibrations.  
Memoirs Amer. Math. Soc. No. 155, 1975.

\bibitem{MQR}
J.P. May (with contributions by F. Quinn, N. Ray, and J. Tornehave). 
$E_{\infty}$ ring spaces and $E_{\infty}$ ring spectra.  
Lecture Notes in Mathematics Vol. 577.  Springer-Verlag 1977.

\bibitem{MayPer}
J.P. May
$E_{\infty}$ spaces, group completions, and permutative categories.    
London Math. Soc.  Lecture Notes Series Vol. 11, 1974, 61--93.

\bibitem{MayPer2}
J.P. May.
The spectra associated to permutative categories.  
Topology 17(1978), 225--228.

\bibitem{IMon}
J.P. May.
The spectra associated to $\sI$-monoids.
Math. Proc. Camb. Phil. Soc. 84(1978), 313--322.

\bibitem{Fib}
J.P. May.
Fibrewise localization and completion. 
Trans. Amer. Math. Soc. 258(1980), 127--146.

\bibitem{MayTalk}
J.P. May.
Thom Thom spectra and other new brave algebras.
Notes for a talk at MIT, Nov. 5, 2007.
Available at
http://www.math.uchicago.edu/~may/Talks/ThomTalk.pdf.

\bibitem{Prequel1}
J.P. May.
What precisely are $E_{\infty}$ ring spaces and $E_{\infty}$ ring spectra?
This volume.

\bibitem{Prequel2}
J.P. May.
The construction of $E_{\infty}$ ring spaces from bipermutative categories.
This volume. 

\bibitem{MS}
J.P. May and J. Sigurdsson.
Parametrized homotopy theory. 
Amer. Math. Soc.Mathematical Surveys and Monographs Vol 132. 2006.

\bibitem{MT}
J.P. May and R. Thomason.  
The uniqueness of infinite loop space machines.  
Topology 17(1978), 205-224.

\bibitem{Quillen}
D. Quillen.
The Adams conjecture.
Topology 10(1971), 67--80.

\bibitem{Quill2} 
D. Quillen.
On the cohomology and $K$-theory of the general
linear groups over a finite field.
Annals Math. 96(1972), 552--586. 

\bibitem{Quinn}
F. Quinn
Surgery on Poincar\'e and normal spaces.
Bull. Amer. Math. Soc. 78(1972), 262--267. 

\bibitem{Ray}
N. Ray.
Bordism $J$-homomorphisms. 
Illinois J. Math. 18(1974), 237--248.

\bibitem{Seg2}
G. Segal. 
Categories and cohomology theories.  
Topology  13(1974), 293--312. 

\bibitem{Serre}
J.P. Serre.
Repr\'esentations lin\'eaires des groupes finis.
Herrmann, 1967.

\bibitem{Sull}
D.P. Sullivan. 
Geometric topology: localization, periodicity and Galois symmetry.
The 1970 MIT notes. Edited and with a preface by Andrew Ranicki. 
$K$-Monographs in Mathematics, 8. Springer, Dordrecht, 2005.

\end{thebibliography}
\end{document}